\documentclass[11pt,a4paper,reqno,dvipsnames]{amsart}
\usepackage{amsfonts}
\usepackage{tikz}
\usepackage{amsmath}
\usepackage{amssymb,stmaryrd}
\usepackage{amsthm}
\usepackage{amsxtra}
\usepackage{bbm}
\usepackage{braket}
\usepackage{comment}
\usepackage{extarrows}
\usepackage{graphicx}
\usepackage{latexsym}
\usepackage{mathrsfs}
\usepackage{yhmath}
\usepackage{nicematrix}
\usepackage{mathtools} 

\let\underbrace\LaTeXunderbrace

\usepackage{float}
\usepackage{booktabs}
\usepackage{xcolor}
\usepackage{verbatim}
\usepackage{listings}
\usepackage{natbib}
\setcitestyle{numbers,square,semicolon}

\usepackage{pgfplots}
\pgfplotsset{compat=1.15}
\usetikzlibrary{arrows}
\usetikzlibrary{decorations.pathreplacing}
\usepackage{parskip}

\usepackage[normalem]{ulem}

\usepackage[pdfborder={0 0 0}]{hyperref} % hyperref must be loaded after natbib
\hypersetup{colorlinks, citecolor=blue, linkcolor=blue, urlcolor=teal}

\usepackage{bm}
\usepackage{tikz}
\usetikzlibrary{graphs}
\usepackage{multicol}
\usepackage{cleveref}

\providecommand{\Cref}{\cref}
\providecommand{\cref}{\cref}

\makeatletter
\newcommand{\refcheckize}[1]{%
    \expandafter\let\csname @@\string#1\endcsname#1%
\expandafter\DeclareRobustCommand\csname relax\string#1\endcsname[1]{%
\csname @@\string#1\endcsname{##1}\wrtusdrf{##1}}%
 \expandafter\let\expandafter#1\csname relax\string#1\endcsname
}
\makeatother
\usetikzlibrary{shapes.geometric}

\usepackage{fullpage}

\makeatother
\newtheorem*{rep@theorem}{\rep@title}
\newcommand{\newreptheorem}[2]{%
\newenvironment{rep#1}[1]{%
 \def\rep@title{#2 \ref{##1}}%
 \begin{rep@theorem}}%
 {\end{rep@theorem}}}
\makeatother

\newreptheorem{theorem}{Theorem}
\theoremstyle{plain}
\newtheorem{thm}{Theorem}[section]
\newtheorem{lem}[thm]{Lemma}
\newtheorem{prop}[thm]{Proposition}
\newtheorem{cor}[thm]{Corollary}

\newtheorem{question}{Question}
\theoremstyle{definition}

\numberwithin{case}{thm}

\numberwithin{subcase}{case}

\theoremstyle{definition}
\newtheorem{defn}[thm]{Definition}
\newtheorem{ex}[thm]{Example}
\newtheorem{notation}[thm]{Notation}
\newtheorem{remark}[thm]{Remark}

\usepackage[shortlabels]{enumitem}
\SetEnumerateShortLabel{a}{\textup{(\alph*)}}
\SetEnumerateShortLabel{A}{\textup{(\Alph*)}}
\SetEnumerateShortLabel{1}{\textup{(\arabic*)}}
\SetEnumerateShortLabel{i}{\textup{(\roman*)}}
\SetEnumerateShortLabel{I}{\textup{(\Roman*)}}

\DeclareMathOperator{\reg}{reg}
\DeclareMathOperator{\init}{in}

\begin{document}

\title{Degree of $h$-polynomials of edge ideals}

% \thanks{2020 {\em Mathematics Subject Classification}. }

% \thanks{Keywords: Edge ideals, regularity, Hilbert series, $h$-polynomial}

\author[Biermann]{Jennifer Biermann} 
\address[J.~Biermann]{Department of Mathematics and Computer Science, Hobart and William Smith Colleges, 300 Pulteney
St. Geneva, NY 14456, USA}
\email{\textcolor{blue}{\href{mailto:biermann@hws.edu}{biermann@hws.edu}}}

\author[Kara]{Selvi Kara} 
\address[S.~Kara]{Department of Mathematics, Bryn Mawr College, Bryn Mawr, PA 19010}
\email{\textcolor{blue}{\href{mailto:skara@brynmawr.edu}{skara@brynmawr.edu}}}

\author[O'Keefe]{Augustine O'Keefe} 
\address[A.~O'Keefe]{Department of Mathematics and Statistics, Connecticut College, 270 Mohegan Avenue Pkwy, New
London, CT 06320, USA}
\email{\textcolor{blue}{\href{mailto:aokeefe@conncoll.edu}{aokeefe@conncoll.edu}}}

\author[Skelton]{Joseph Skelton} 
\address[J.~Skelton]{School of Mathematical and Statistical Sciences, Clemson University, Martin Hall, Clemson, SC 29634, USA}
\email{\textcolor{blue}{\href{mailto:jwskelt@clemson.edu }{jwskelt@clemson.edu }}}

\author[Sosa Castillo]{Gabriel Sosa Castillo} 
\address[G.~Sosa Castillo]{Department of Mathematics, Colgate University,
Hamilton, NY, 13346, USA}
\email{\textcolor{blue}{\href{mailto:gsosacastillo@colgate.edu}{gsosacastillo@colgate.edu}}}

\begin{abstract}
In this paper, we investigate the degree of $h$-polynomials of edge ideals of finite simple graphs. In particular, we provide combinatorial formulas for the degree of the $h$-polynomial for various fundamental classes of graphs such as paths, cycles, and bipartite graphs. To the best of our knowledge, this marks the first investigation into the combinatorial interpretation of this algebraic invariant. Additionally, we characterize all connected graphs in which the sum of the Castelnuovo-Mumford regularity and the degree of the  $h$-polynomial of an edge ideal reaches its maximum value, which is the number of vertices in the graph.

\end{abstract}

\maketitle

\section{Introduction}

Edge ideals of graphs provide a fertile ground for uncovering combinatorial expressions for algebraic invariants of monomial ideals, thereby making abstract algebraic concepts more accessible. There has thus been extensive research into expressing  or bounding fundamental algebraic invariants—such as the Castelnuovo-Mumford regularity and projective dimension —in terms of graph invariants (see, for example, \citep{ProjDim, ha2008monomial,katzman2006characteristic, RegMatch}). In this paper, we continue this trend by investigating  the degree of the $h$-polynomial of edge ideals, a fundamental yet less explored algebraic invariant.

Let \( G = (V(G), E(G)) \) be a finite simple graph with vertex set \( V(G) = \{x_1, \ldots, x_n\} \) and edge set \( E(G) \). The edge ideal of \( G \), denoted by \( I(G) \), is defined as the monomial ideal
\[ I(G) = (x_i x_j \mid \{x_i, x_j\} \in E(G)) \subseteq R = k[x_1, \ldots, x_n]. \]
The \( h \)-polynomial of \( R/I(G) \), denoted by \( h_{R/I(G)}(t) \), is  related to the Hilbert series.  It is well-known that Hilbert series of \( R/I(G) \) can be expressed as a rational function  (see \citep[Theorem 13.2]{matsumura1987}) and the numerator of this (reduced) rational function is called the \( h \)-polynomial of \( R/I(G) \).

There has been recent interest in studying  the degree of \( h_{R/I(G)}(t) \), denoted by $\deg h_{R/I(G)}(t)$, in relation to fundamental invariants of \( R/I(G) \) such as the dimension, projective dimension, depth, and regularity. The recent research has focused on determining the possible values of these invariants for particular classes of graphs \citep{hibi2021homological, hibi2021regularity} and the existence of a graph \( G \) for given values of these invariants \citep{IndMatch}. Additionally, researchers have studied the relationships between the degree of the \( h \)-polynomial and these invariants \citep{hibi2019regularity}. Moreover, similar studies have been undertaken for  monomial ideals in \citep{hibi2018regularity} and toric ideals of graphs in \citep{favacchio2020regularity}.

In this paper, we take a different approach by solely focusing on computing the degree of the \( h \)-polynomial for edge ideals. To the best of our knowledge, our work is the first in-depth study of this  invariant. Our primary goal here is to provide formulas for \(\deg h_{R/I(G)}(t)\) based on the combinatorial data of \( G \) for various graph classes. In particular, we are interested in expressing \(\deg h_{R/I(G)}(t)\) in terms of the independence number of a graph.

\begin{remark}\label{rmk:deg_alpha}
The relationship between the degree of the Hilbert polynomial and combinatorial invariants of the graph arises from the following expression for the Hilbert series of \( R/I(G) \):
\[
H_{R/I(G)}(t) = \sum_{i=0}^{\alpha(G)} \frac{f_{i-1} t^i}{(1-t)^i}
\]
where $H_{R/I(G)}(t)$ denotes the Hilbert series of $R/I(G)$, \( f_{i-1} \) denotes the number of independent sets of cardinality \( i \) in \( G \) (by convention $f_{-1} = 1$), and \( \alpha(G) \) represents the size of the largest independent set in \( G \) (see \citep[Theorem 6.2.1]{herzog2011monomial}). It follows from this expression that 
    \[
    \deg h_{R/I(G)}(t) \leq \alpha(G).
    \]
\end{remark}
In light of this upper bound on the degree, one of our objectives  is to identify classes of graphs for which \(\deg h_{R/I(G)}(t) = \alpha(G)\).

In this paper, we provide explicit formulas for the degree of the $h$-polynomial for paths, cycles and complete bipartite graphs in terms of the independence number. In addition, we provide a classification of connected bipartite graphs where  $\deg h_{R/I(G)}(t)  = \alpha(G)$. We now list the main results of this paper.

\begin{reptheorem}{thm:path_cycle}
Let $C_n$ denote the cycle on $n$-vertices and $P_n$ denote the path on $n$-vertices. 
\begin{enumerate}
    \item The degree of the $h$-polynomial of $R/I(C_n)$ for $n\geq 3$ is 
$$  \deg h_{R/I(C_n)}(t)  = \alpha(C_n)$$
where $\displaystyle \alpha(C_n) =\left\lfloor \frac{n}{2} \right\rfloor$.
    \item The degree of the $h$-polynomial of $R/I(P_n)$ for $n\geq 1$ is 
$$ \deg h_{R/I(P_n)}(t) = \begin{cases}
        \alpha(P_n) &: n \equiv  0,2 \pmod 3,\\
       \alpha(P_n)-1 &: n \equiv 1 \pmod 3
    \end{cases}$$
    where $ \displaystyle \alpha(P_n)= \left\lceil \frac{n}{2} \right\rceil$.
\end{enumerate}
\end{reptheorem}

Our result on bipartite graphs is presented based on notation introduced in Section \ref{sec:bipartite}. To make the statement of our result more accessible, we briefly introduce some necessary notation here, and refer the reader to Section \ref{sec:bipartite} for details.

\begin{reptheorem}{thm:bipartite}
 Let $G$ be a connected bipartite graph with bipartition $(U,V)$ such that $|U|\geq |V|$.  Let $V_W$ denote the collection of whiskered vertices (i.e. vertices adjacent to a leaf) in $V$ and $\overline{V_W}$ denote the complement of $V_W$ in $V.$ Similarly, let $U_L$ be a set of leaf vertices in $U$ with the property that there is exactly one element adjacent to each vertex of $V_W$ and $\overline{U_L}$ be the complement of $U_L$ in $U.$ 
\begin{enumerate}[a]
        \item If every vertex in $V$ is whiskered,  then  $\deg h_{R/I(G)} (t)= \alpha(G)=|U|$.
        \item If there exist vertices in $V$ that are not whiskered, then  $\deg h_{R/I(G)} (t)= \alpha(G)$ if and only if  the number of even cardinality non-empty subsets of $\overline{U_L}$ whose set of neighbors is exactly $\overline{V_W}$ is not equal to the number of odd cardinality subsets of $\overline{U_L}$ with the same property. 
    \end{enumerate}
\end{reptheorem}

A nice consequence of \Cref{thm:bipartite} is that the degree of a complete bipartite graph is $\alpha(G)$.

The other focus of this paper is to study the relationship between the degree of the $h$-polynomial and \( \reg R/I(G) \), the regularity of \( R/I(G) \). Our interest stems from the following inequality \citep[Theorem 13]{hibi2019regularity}:
\[
\reg R/I(G) + \deg h_{R/I(G)}(t) \leq |V(G)|.
\]
In particular, we are interested in determining when this bound is sharp. One of the main results of this paper, stated below, characterizes all connected finite simple graphs for which the equality is achieved.

\begin{reptheorem}{thm:reg_deg=n}
  Let $G$ be a connected graph. Then $$\reg R/I(G) + \deg h_{R/I(G)}(t) =|V(G)|$$  if and only if $G$ is  a Cameron-Walker graph with no pendant triangles.    
\end{reptheorem}

The structure of the paper is as follows: We collect the necessary background on graph theory and algebra in \Cref{sec:prelim}. In \Cref{sec:max_degree}, we present a graph theoretical (necessary and sufficient) condition for \(\deg h_{R/I(G)}(t)  = \alpha(G)\). We provide formulas for \(\deg h_{R/I(G)}(t)\) for paths and cycles in \Cref{sec:path_cycles}. \Cref{sec:bipartite} is dedicated to classifying bipartite graphs where \(\deg h_{R/I(G)}(t) = \alpha(G)\). Finally, in \Cref{sec:classification}, we identify all connected graphs \( G \) for which the sum of the degree of the \( h \)-polynomial and the regularity is equal to \(|V(G)|\).

\begin{remark}
The term degree is an overloaded one within mathematics. Throughout this paper, we may refer to the degree of the $h$-polynomial of $R/I(G)$ as the degree of $R/I(G)$. 
\end{remark}

\section{Preliminaries}\label{sec:prelim}

In this section, we collect fundamental definitions and results that are relevant to this paper. A reader who is familiar with this background material can skip this section. We start with the algebraic concepts and then introduce the related notions from graph theory.

\begin{defn}
    The \textbf{Hilbert function} of the graded ring $R/I(G)$, $HF_{R/I(G)}:\mathbb{Z}_{\geq 0}\rightarrow \mathbb{Z}_{\geq 0},$ is defined by 
    \[
        HF_{R/I(G)}(i)=\dim_k(R/I(G))_i
    \]
    where $\dim_k(R/I(G))_i$ is the $k$-vector space dimension of the $i^\text{th}$ graded piece of $R/I(G)$. 
    
    The \textbf{Hilbert series} of $R/I(G)$, denoted by $H_{R/I(G)}(t)$, is defined to be
\[ H_{R/I(G)}(t) = \sum_{i \in \mathbb{Z}} HF_{R/I(G)}(i) t^i. \]
\end{defn}

Via the following well-known theorem we see that we can compute the Hilbert function of $R/I(G)$ by counting \emph{standard monomials}, i.e. those monomials not in $\init_<I(G)=I(G)$. 

\begin{prop}\citetext{\citealp[Proposition 1.1]{Sturmfels1996}} 
Let $I$ be an ideal in the polynomial ring $R=k[\mathbf{x}]$. The standard monomials of $I$ form a $k$-vector space basis for $R/I.$
\end{prop}

\begin{defn}
 A graded minimal free resolution of $R/I(G)$ is an exact sequence of the form
$$
0\longrightarrow \bigoplus_{j \in \mathbb{Z}} R(-j)^{\beta_{p,j}} \longrightarrow \bigoplus_{j \in \mathbb{Z}} R(-j)^{\beta_{p-1,j}} \longrightarrow \cdots \longrightarrow 
\bigoplus_{j \in \mathbb{Z}} R(-j)^{\beta_{0,j}} \longrightarrow R/I(G) \longrightarrow 0.
$$
 The exponents $\beta_{i,j} = \beta_{i,j}(R/I(G))$ are invariants of $R/I(G)$ called the graded Betti numbers of $R/I(G)$.  The Castelnuovo-Mumford regularity, denoted by $\reg (R/I(G))$, is  defined in terms of the non-zero graded Betti numbers of $R/I(G)$:
 \[
    \reg (R/I(G)) = \max \{ j-i : \beta_{i,j} (R/I(G)) \neq 0\}.
 \] 
\end{defn}

  One can bound the regularity of $R/I(G)$ in terms of combinatorial invariants of $G$. In what follows, we recall such graph invariants. 

  \begin{defn}
     A \emph{matching} of $G$ is a collection of pairwise disjoint edges of $G$.  The \emph{matching number} of $G$ is the maximum size of a matching of $G$ and it is denoted by $\mu(G)$.  
     
     A matching $\mathcal{M}$ is called an \emph{induced matching} if the set of edges of the induced graph on the vertices of $\mathcal{M}$  is equal to $\mathcal{M}$.  The \emph{induced matching number} of $G$ is the maximum size of an induced matching of $G$ and it is denoted by $\nu(G)$.

     It is immediate that $\nu(G)\leq \mu(G)$. Graphs for which $\nu(G)=\mu(G)$ are called \emph{Cameron-Walker graphs}.
\end{defn}

\begin{remark}
    From \citetext{\citealp[Lemma 2.2]{katzman2006characteristic}} and \citetext{\citealp[Theorem 6.7]{ha2008monomial}}  one has
    $$\nu(G) \leq \reg (R/I(G)) \leq \mu(G)$$
    for any finite simple graph $G$.
    % where $\nu(G)$ is the induced matching number and $\mu(G)$ is the matching number of $G$.    
\end{remark}

A characterization of graphs where $\reg (R/I(G)) = \mu(G)$ is given in  \citep{RegMatch} and we recall its statement below. 

\begin{thm}\label{thm:reg_match}\citetext{\citealp[Theorem 11]{RegMatch}}
 Let $G$ be a connected graph. Then $\reg R/I(G) =\mu(G)$ if and only if $G$ is a Cameron-Walker graph or $G$ is a 5-cycle.   
\end{thm}

The main graph invariant that is relevant to our discussion of the degree of the $h$-polynomial is the independence number. We therefore recall the notion of independent sets and independence number of a graph here.

\begin{defn}
   Let $G$ be a finite simple graph. A collection of vertices $\mathcal{C}\subseteq V(G)$ is called an \textbf{independent set} of $G$ if $\{x,y\} \notin E(G)$ for any $x,y \in \mathcal{C}$. The \textbf{independence number} of  $G$ is the maximum size of an independent set and it is denoted by $\alpha(G)$. When $G$ is understood we will simply denote $\alpha(G)$ by $\alpha$. 
\end{defn}

The following result which provides an upper bound for the sum of the matching number and independence number of a graph follows from one of the Gallai identities (see for example \citep[Lemma 1.0.1]{LovaszPlummer1986}). 

\begin{thm}\label{thm:konig}
    Let $G$ be a  graph on $n$ vertices. Then $\alpha(G)+\mu(G) \leq n$.
\end{thm}

    The final graph theory concepts we recall are the neighbors of a vertex and the notion of leaves.
\begin{defn}
    Let $G$ be a finite simple graph and $u, v \in V(G)$. If $\{u,v\}\in E(G)$, we say $u$ and $v$ are \textbf{neighbors} in $G$. The set of all neighbors of a vertex $u$, denoted $N(u),$ is called its \textbf{neighborhood}, ie. $N(u)=\{v\in V(G)\,:\,\{u,v\}\in E(G)\}$.
    
    A vertex with only one neighbor in $G$ is called a \textbf{leaf vertex}, or merely \textbf{leaf}. A neighbor of a leaf is said to be \textbf{whiskered}.
\end{defn}

Lastly, we present the following combinatorial identity which will be useful in \Cref{sec:max_degree} for determining when \(\deg h_{R/I(G)}(t) = \alpha(G)\). Its proof is provided for completeness.

\begin{lem}\label{importantcombination}
Let $k$ and $d$ be non-negative integers such that  $d\geq 1$. Then
\begin{equation}\label{binomidentity}
\binom{d-1}{k}=\sum_{i=0}^k (-1)^i\binom{d+k-i}{k-i}\binom{k+1}{i}.
\end{equation}

\end{lem}

\begin{proof}
We count the positive integer solutions to the equation
\begin{equation}\label{integersolutions} \hskip0.5cm x_1+x_2+\dots +x_{k+1}=d
\end{equation}
in two ways.
On one hand, it is straightforward to see that there are $\binom{d-1}{k}$  many positive integer solutions to Equation \ref{integersolutions}.  On the other hand, one can also count the non-negative integer solutions and remove those which contain zeros using an inclusion/exclusion argument.  First note that there are $\binom{d+k}{k}\binom{k+1}{0}$ many non-negative integer solutions to \Cref{integersolutions}. For $0<i\leq k$, a naive way to count the number of non-negative solutions to  \Cref{integersolutions} with at least $i$ of the variables being $0$ would be to first choose $i$ of the variables to be zero and then the number of non-negative integer solutions to the equation in the remaining $k+1-i$ variables is $\binom{d+k-i}{k-i}$.  This results in the following expression
\[
    \binom{d+k-i}{k-i}\binom{k+1}{i}
\]
which over-counts those solutions that have more than $i$ variables valued zero. Applying the Principle of Inclusion-Exclusion to deal with the overlap leads to the alternating sum on the right side of \Cref{binomidentity}.\qedhere

\end{proof}

\section{Maximum Degree via Independent Sets}\label{sec:max_degree}

In this section, we analyze when the degree of the $h$-polynomial of $R/I(G)$ achieves its maximum value $\alpha(G)$, the independence number of $G$. We adopt the following notation for the remainder of this paper.

\begin{notation}\label{notationforHS} 
We use the notation $g_j(G)=g_j$ to denote the number of independent sets of size $j$ in a  graph  $G$ for $1\leq j \leq \alpha(G)$. Equivalently,
$$g_j = g_j(G)=|\{ \mathcal{C} \subseteq V(G) :  \mathcal{C} \text{ is an independent set of } G \text{ and } |\mathcal{C}|=j\}|$$
Let $g(G)=g$ denote the alternating sum of $g_1, g_2, \ldots, g_{\alpha(G)}$ where
 $$g=\sum_{j=1}^{\alpha(G)} (-1)^{j-1} g_j.$$
\end{notation}

\begin{thm}\label{Thm:alternate_sum} 
Let $G$ be a finite simple graph. Then $\deg h_{R/I(G)}(t)=\alpha(G) $ if and only if $g\neq 1$.
%$g_1- g_2+\cdots+ (-1)^{\ell-1}g_{\ell} \neq 1$.
\end{thm}
\begin{proof}
 Let  $\alpha=\alpha(G)$.  First, observe that  the support of a standard monomial of $R/I(G)$ coincides with an independent set of $G$. In particular, for an independent set $\mathcal{C}$ of size $j$, there are $ \displaystyle\binom{d-1}{j-1}$ many standard monomials of degree $d$ with support  $\mathcal{C}$. Thus,  %we have
 %$$\text{HF}_{R/ I(G)}(d)= \displaystyle g_j\binom{d-1}{j-1}.$$ Hence, 
 the Hilbert series of $R/I(G)$ can be expressed as
\begin{equation}\label{eq:HS_1}
    H_{R/I(G)} (t)= 1+\sum_{d=1}^{\infty} \text{HF}_{R/ I(G)}(d)t^d= 1+\sum_{d=1}^{\infty}\left(\sum_{j=1}^{\alpha} g_j\binom{d-1}{j-1} \right)t^d.
\end{equation} 
\textbf{Claim:} The Hilbert series can be re-written as follows
\begin{equation}\label{eq:HS_2}
H_{R/I(G)} (t) =\frac{\left(1-\sum_{s=0}^{\alpha-1}D_s\right)(1-t)^{\alpha}+\sum_{s=0}^{\alpha-1}D_s(1-t)^{\alpha-s-1}}{(1-t)^{\alpha}}
\end{equation}
 where 
 \begin{equation}\label{eq:D_s} D_s=\sum_{j=s+1}^{\alpha} (-1)^{j-1-s}g_j\binom{j}{j-1-s}, ~~ 0\leq s\leq \alpha-1.
 \end{equation}

 Observe that the constant term of the numerator in \Cref{eq:HS_2} as a polynomial in $(1-t)$ is nonzero, namely $D_{\alpha-1}= g_{\alpha}\neq0$. Therefore, if the claim is true,  $\deg h_{R/I(G)}(t)=\alpha$ if and only if $1-\sum_{s=0}^{\alpha-1}D_s\neq 0.$ We can then rewrite $\sum_{s=0}^{\alpha-1}D_s$ as follows by first reversing the order of summation and then applying a consequence of the Binomial Theorem, namely that $\sum_{s=0}^{j}(-1)^s\binom{j}{s}= 0$.
\begin{align*}
\sum_{s=0}^{\alpha-1}D_s&=\sum_{s=0}^{\alpha-1}\left[\sum_{j=s+1}^{\alpha} (-1)^{j-1-s}g_j\binom{j}{j-1-s}\right]\\
& =\sum_{j=1}^{\alpha}g_j \left[ \sum_{s=0}^{j-1}(-1)^s\binom{j}{s} \right]\\
&=\sum_{j=1}^{\alpha} (-1)^{j-1}g_j=g.\end{align*}
 We then have that $\deg h_{R/I(G)}(t)=\alpha(G)$ whenever $g\neq 1$, thus proving the lemma.

\textbf{Proof of Claim:} To see that (\ref{eq:HS_2}) follows from from (\ref{eq:HS_1}) we use Lemma \ref{importantcombination} to break the binomial factor $\binom{d-1}{j-1}$ into a summation to obtain
$$H_{R/I(G)}(t)= 1+\sum_{d=1}^\infty\left(\sum_{j=1}^\alpha\left(\sum_{i=0}^{j-1} (-1)^i g_j\binom{d+j-1-i}{j-1-i}\binom{j}{i}\right) t^d\right).$$

We then substitute $s = j-1-i$ to get

\begin{align*}
H_{R/I(G)}(t) &= 1+ \sum_{d=1}^\infty \left(\sum_{j=1}^\alpha \left ( \sum_{s = 0}^{j-1} (-1)^{j-1-s}g_j \binom{d+s}{s}\binom{j}{j-1-s} \right )t^d \right)\\
&= 1+\sum_{j=1}^\alpha \left( \sum_{s=0}^{j-1} \left( (-1)^{j-1-s} g_j \binom{j}{j-1-s}\sum_{d=1}^\infty \binom{d+s}{s} t^d \right) \right) .
\end{align*}

Reversing the order of the summations in $j$ and $s$ yields
\begin{align*}
H_{R/I(G)}(t) &= 1+\sum_{s=0}^{\alpha-1}\left(\sum_{j=s+1}^{\alpha}\left( (-1)^{j-1-s}g_j\binom{j}{j-1-s} \sum_{d=1}^\infty\binom{d+s}{s}t^d\right)\right)\\
 &= 1+\sum_{s=0}^{\alpha-1}\left( \left(\sum_{d=1}^\infty\binom{d+s}{s}t^d\right) \cdot \left(\sum_{j=s+1}^{\alpha}(-1)^{j-1-s}g_j\binom{j}{j-1-s} \right)\right).
 \end{align*}
We are then able to substitute $D_s$ as described above to obtain
\begin{align*}
H_{R/I(G)} (t) 
&= 1+\sum_{s=0}^{\alpha-1}D_s \left[ \sum_{d=1}^{\infty}\binom{d+s}{s} t^d\right]\\
&= 1-\sum_{s=0}^{\alpha-1}D_s+\sum_{s=0}^{\alpha-1}\frac{D_s}{(1-t)^{s+1}}\\
&= \frac{\left(1-\sum_{s=0}^{\alpha-1}D_s\right)(1-t)^{\alpha}+\sum_{s=0}^{\alpha-1}D_s(1-t)^{\alpha-s-1}}{(1-t)^{\alpha}}
% &= \frac{\left(1-\sum_{s=0}^{\ell-1}D_s\right)(1-t)^{\ell}+D_0(1-t)^{\ell-1}+D_1(1-t)^{\ell-2} + \cdots+ D_{\ell-2}(1-t)+ D_{l-1}}{(1-t)^{\ell}}
\end{align*}
where the second to last equality is due to the following $$\displaystyle 1+\sum_{d=1}^{\infty}\binom{d+s}{s} t^d=\frac{1}{(1-t)^{s+1}}$$
which is the Hilbert series of the polynomial ring $k[x_0,x_1,\ldots, x_s]$.  
\end{proof}

\begin{ex}
Consider the following triangle star graphs $G$ and $H$ with 2 and 3 triangles, respectively, in \Cref{fig:ex3.3}. 
\begin{figure}[h]
\centering
\begin{tikzpicture}[scale = .5,line cap=round,line join=round,>=triangle 45,x=1cm,y=1cm]
\clip(-4.60520298625337,-4.307995749895294) rectangle (5.063704402188539,2.412813004067189);
\draw [line width=1pt] (-3,2.16)-- (-3,-2.16);
\draw [line width=1pt] (-3,-2.16)-- (0,0);
\draw [line width=1pt] (-3,2.16)-- (0,0);
\draw [line width=1pt] (0,0)-- (3,2.16);
\draw [line width=1pt] (3,2.16)-- (3,-2.16);
\draw [line width=1pt] (3,-2.16)-- (0,0);
\draw (-3.3,2) node[anchor=east]{$G:$};
\begin{scriptsize}
\draw [fill=black] (-3,2.16)  circle (5pt);
\draw [fill=black] (-3,-2.16) circle (5pt);
\draw [fill=black] (0,0) circle (5pt);
\draw [fill=black] (3,2.16) circle (5pt);
\draw [fill=black] (3,-2.16) circle (5pt);
\end{scriptsize}
\end{tikzpicture}
%\hspace{-.5in}
\begin{tikzpicture}[scale = .5,line cap=round,line join=round,>=triangle 45,x=1cm,y=1cm]
\clip(-4.60520298625337,-4.307995749895294) rectangle (5.063704402188539,3.22813004067189);
\draw [line width=1pt] (0,0)-- (-1.44,3);
\draw [line width=1pt] (0,0)-- (1.44,3);
\draw [line width=1pt] (2.88,0)-- (1.44,3);
\draw [line width=1pt] (-2.88,0)-- (-1.44,3);
\draw [line width=1pt] (-2.88,0)-- (0,0);
\draw [line width=1pt] (2.88,0)-- (0,0);
\draw [line width=1pt] (0,0)-- (-1.44,-3);
\draw [line width=1pt] (0,0)-- (1.44,-3);
\draw [line width=1pt] (1.44,-3)-- (-1.44,-3);
\draw (-3.3,2) node[anchor=east]{$H:$};
\begin{scriptsize}
\draw [fill=black] (-1.44,3) circle (5pt);
\draw [fill=black] (1.44,3) circle (5pt);
\draw [fill=black] (0,0) circle (5pt);
\draw [fill=black] (2.88,0) circle (5pt);
\draw [fill=black] (-2.88,0) circle (5pt);
\draw [fill=black] (-1.44,-3) circle (5pt);
\draw [fill=black] (1.44,-3) circle (5pt);
\end{scriptsize}
\end{tikzpicture}
\caption{$\deg h_{R/I(G)}(t)\neq\alpha(G)$ and $\deg h_{R/I(H)}(t)=\alpha(H)$}
\label{fig:ex3.3}
\end{figure}
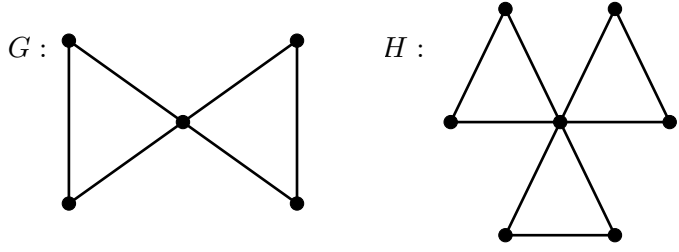

A simple computation gives $\deg h_{R/I(G)}(t)=1$ and $\deg h_{R/I(H)}(t)=3$. By inspection, we can see $\alpha(G) = 2$ and $\alpha(H)=3$. Finally, counting independent sets we can see $g(G) = g_1 - g_2 = 5 - 4 = 1$, as predicted by Theorem \ref{Thm:alternate_sum}, and $g(H) = g_1 - g_2+ g_3 = 7 - 12 + 8 = 3$.
\end{ex}

\begin{cor}\label{DegreeintermsofD}
 If $g=1$, then $\deg h_{R/I(G)}(t)=\alpha(G) - d'-1$ where $$d'= \min \{s : D_s\neq0\}$$ for $0\leq s\leq \alpha(G)-1$.
\end{cor}

\section{Degrees of Paths and Cycles}\label{sec:path_cycles}

In this section, we provide formulas for the degree of the $h$-polynomial for edge ideals of paths and cycles.

Let $P_n$ and $C_n$ denote the path and cycle on $n$ vertices, respectively. Label the vertices of $P_n$ and $C_n$ as $v_1, v_2, \ldots, v_n$ such that $\{v_j, v_{j+1}\}$ is an edge of $P_n$ and $C_n$ for all $1 \leq j \leq n-1$, and $\{v_1, v_n\}$ is an edge of $C_n$. Recall from Notation \ref{notationforHS} that $g_i$ denotes the number of independent sets in size $i$ for a given graph $G$. We begin this section by presenting closed formulas for $g_i$ where $G$ is $P_n$ or $C_n$.

\begin{lem}\label{gformulaforPath}
 Let $p_{n,i}$ denote the number of independent sets of size $i$ for the path $P_n$. Then  
  $$ \displaystyle p_{n,i}= \begin{cases}
        \binom{n-i+1}{i} & \text{if }  n\geq 2i-1,\\
        0 &  \text{otherwise}.
    \end{cases}$$
\end{lem}

\begin{proof}
We will proceed by induction on $n$ with $0\leq i \leq n$. The statement is clear when $n=1$ and $n=2$, or when $i=0$. 

An independent set of size $i>0$ for $P_n$ with $n\geq 3$ either contains the leaf $v_n$ or does not. Hence the number of independent sets containing $v_n$ is $p_{n-2,i-1}$ while the number of independent sets that do not is $p_{n-1,i}$. Thus we have the recurrence relation $p_{n,i}=p_{n-1,i}+p_{n-2,i-1}$, and there are now three cases to consider based on the relation of $n$ with $2i-1$. 

If $n\geq 2i$, then
    \[ 
        p_{n,i} =\binom{n-i}{i}+\binom{n-i}{i-1}=\binom{n-i+1}{i}.
    \]
If $n=2i-1$, then 
    \[
        p_{n,i}=0+\binom{i-1}{i-1}=\binom{i}{i}=\binom{n-i+1}{i}.
    \]
Finally, if $n <2i-1$, we have $\displaystyle p_{n,i}=0+0=0$. 
\end{proof}

\begin{lem}\label{gformulaforCycle}
    Let $c_{n,i}$ denote the number of independent sets of size $i$ for the cycle $C_n$ for $n\geq 3$. Then 
    $$ \displaystyle c_{n,i}= \begin{cases}
        \frac{n}{i}\binom{n-i-1}{i-1} & \text{if }  n\geq 2i,\\
        0 &  \text{otherwise.}
    \end{cases}$$
\end{lem}

\begin{proof}
Note that $c_{n,i} = 0$ for $n<2i$ so we assume $n\geq 2i$.  It is immediate that $c_{n,1}=n$ for $n\geq 3$. Suppose $i\geq 2$. For a vertex $v_j$ on the cycle, consider  an independent set $\mathcal{C}$ of $C_n$ such that $v_j \in \mathcal{C}$ and $|\mathcal{C}|=i$. The set $\mathcal{C} \setminus \{v_j\}$ is an independent set of a path on $(n-3)$-vertices (because $v_{j-1}$ and $v_{j+1}$ (modulo $n$) are not in $\mathcal{C}$). 

 So, the number of independent sets of size $i$ of $C_n$ containing the vertex $v_j$ is equal to $p_{n-3,i-1}$. Since there are $n$ vertices in $C_n$ we multiply by $n$ to get $np_{n-3,i-1}$. This counts each independent set of size $i$ in $C_n$ a total of $i$ times, so $i c_{n,i}=n p_{n-3,i-1}$.  Thus, we have
$$c_{n,i}  = \frac{n}{i} (p_{n-3,i-1})=  \begin{cases} \displaystyle \frac{n}{i}\binom{n-i-1}{i-1} &\text{if } n\geq 2i, \\ 0 & \text{otherwise }\end{cases}$$
 by \Cref{gformulaforPath}.
\end{proof}

It is known that $\alpha(P_n) = \left\lceil \frac{n}{2} \right\rceil$ and $\alpha(C_n) = \left\lfloor \frac{n}{2} \right\rfloor$. Additionally, these values can be derived as corollaries of Lemmas \ref{gformulaforPath} and \ref{gformulaforCycle}.

As discussed in \Cref{sec:max_degree}, one can deduce the degrees of the $h$-polynomials of   $R/I(P_n)$ and $R/I(C_n)$ from the alternating sums of $p_{n,i}$ and $c_{n,i}$. In the next lemma, we provide the values of these sums using Lemmas \ref{gformulaforPath} and \ref{gformulaforCycle}.  For ease of notation, let $$ p_n =  \displaystyle \sum_{j=1}^{\lceil \frac{n}{2}\rceil} (-1)^{j-1}  p_{n,j} \text{ and } c_n =  \displaystyle \sum_{j=1}^{\lfloor \frac{n}{2}\rfloor} (-1)^{j-1}  c_{n,j}.$$

\begin{lem}\label{Sumofgpath_cycle} 
For $n\geq 1$, 
\[ \displaystyle p_n =\begin{cases}  2 & \text{if } n \equiv 2,3 \pmod{6}, \\ 1 & \text{if } n \equiv 1,4 \pmod{6}, \\ 0 &\text{if }  n \equiv 0,5\pmod{6}. \end {cases}\]
For $n\geq 3$, 
\[ \displaystyle c_n =\begin{cases}  -1 & \text{if } n \equiv 0 \pmod{6}, \\ 0 &\text{if } n \equiv 1,5 \pmod{6}, \\ 2 &\text{if } n \equiv 2,4 \pmod{6}, \\ 3 & \text{if } n \equiv 3 \pmod{6}. \end {cases}\]
\end{lem}

\begin{proof}
In order to find the values of $p_n$ and $c_n$, we first show that $p_k=p_{k+6}$ for  $k\geq 1$ and $c_\ell =c_{\ell+6}$ for $\ell\geq 3$. Then, the statement follows by computing the values of $p_n$  and $c_n$ for small  $n$.

As the first step, consider the successive differences of $p_{2m+2}, p_{2m+1}$ and $p_{2m}$ for $m\geq1$. In particular, using \Cref{gformulaforPath} and the following binomial identity $$ \displaystyle \binom{n+1}{i+1}-\binom{n}{i+1}=\binom{n}{i},$$ we get
\begin{align*}
p_{2m+1}-p_{2m} &= \sum_{j=1}^{\lceil \frac{2m+1}{2} \rceil} (-1)^{j-1} p_{2m+1,j} -\sum_{j=1}^{\lceil \frac{2m}{2} \rceil} (-1)^{j-1} p_{2m,j}\\
    &= \sum_{j=1}^{m+1} (-1)^{j-1} \binom{2m+1-j+1}{j} -\sum_{j=1}^{m} (-1)^{j-1} \binom{2m-j+1}{j}\\
    &= (-1)^m +\sum_{j=1}^{m} (-1)^{j-1} \binom{2m+1-j}{j-1}\\
    &= (-1)^m +\sum_{j=0}^{m-1} (-1)^{j} \binom{2m-j}{j}\\
    &= (-1)^m +(-1)^0\binom{2m}{0}+ \sum_{j=1}^{m} (-1)^{j} \binom{2m-j}{j} - (-1)^m\\
    & = 1-p_{2m-1}.
\end{align*}
Similarly, one can verify that  $p_{2m+2}-p_{2m+1}=1-p_{2m}$. 
 
These identities imply that regardless of parity $p_{k+2}-p_{k+1}=1-p_{k} \text{ for } k\geq 1.$ Then, by adding  $p_{k+3}-p_{k+2}=1-p_{k+1}$ and $p_{k+2}-p_{k+1}=1-p_{k}$, we get 
    \begin{equation}\label{eq:pk_recursion}
    p_k+p_{k+3}=2 \text{ for }k\geq 1
    \end{equation}
which in turn implies
    \[
    p_k=p_{k+6}\text{ for }k\geq 1.
    \]  
 Thus to find all values of $p_n$ it suffices to compute $p_1, p_2$ and $p_3$.  Using \Cref{gformulaforPath}, we get $p_1=1$ and $p_2=p_3=2$.  Hence, we obtain the value of $p_n$ for each $n$ as given in the statement of the lemma.

For $c_n$, we use \Cref{gformulaforCycle} and the same arguments used for $p_n$. 
% We will also need the following identity, which is straightforward to verify.
% \[
% \frac{1}{j}\binom{n-j-1}{j-1} =\frac{1}{n-j}\binom{n-j}{j}
% \]
As for $p_n$, one can verify that $ c_{\ell+2}-c_{\ell+1}=1-c_{\ell}$ for $\ell\geq 3$. In addition, $c_\ell+c_{\ell+3}=2$ and $c_\ell=c_{\ell+6}$ for $\ell\geq 3$. Thus, one can conclude the values of $c_n$ for each $n$ by only evaluating $c_n$ for $n=3,4$ and $5$. A quick computation shows that $c_3=3, c_4=2, c_5=0$. This completes  the proof for $c_n$.
\end{proof}

We are now ready to present the formulas for the degrees of the \( h \)-polynomials of paths and cycles. However, before we proceed, we introduce the following auxiliary lemma, which will be useful in providing the explicit formula for $\deg h_{R/I(P_n)}(t)$.

\begin{lem}\label{D0Values}
Let $D_0$ be defined for $P_n$ as in \Cref{eq:D_s} from the proof of \Cref{Thm:alternate_sum}. As the value of $D_0$ depends on $n$, refer to the sum as $T_n$, i.e.
\[
    T_n= D_0= \sum_{j=1}^{\left\lceil \frac{n}{2} \right\rceil} (-1)^{j-1}jp_{n,j}
\]
where $p_{n,j}$ is as defined in \Cref{gformulaforPath}. 
Then we have 
\begin{equation}\label{eq:T_n}
    T_n=\begin{cases}
        (-1)^k (k+1) & \text{if } n=3k+1 \text{ or } n=3k+3,\\
        (-1)^k 2 (k+1) & \text{if } n=3k+2.
        \end{cases}
\end{equation}
\end{lem}

\begin{proof}
We begin by rewriting the summation defining $T_n$ using \Cref{gformulaforPath}
%and the fact that $\alpha(P_n)=\lceil n/2\rceil$. 
$$T_n =\sum_{j=1}^{\left\lceil \frac{n}{2} \right\rceil} (-1)^{j-1} j\binom{n-j+1}{j}.$$

Using a similar strategy as in the proof of \Cref{Sumofgpath_cycle}, we consider successive differences of $T_{2m+2}, T_{2m+1}$ and $ T_{2m}$ for $m\geq 1$ to obtain the recurrence relation 
    \[
        T_{n+2}-T_{n+1}+T_n=1-p_n, \text{ for } n\geq 1.
    \]
Adding the corresponding expressions of the above recursion with first terms equal to $T_{n+6}$, $T_{n+5}$, $T_{n+3}$ and $T_{n+2}$ one obtains
    \[
        T_n+2T_{n+3}+T_{n+6} = 4-p_{n+4}-p_{n+3}-p_{n+1}-p_n.
    \]
 Then applying \Cref{eq:pk_recursion} gives 
    \begin{equation}\label{eq:Tn_recursion}
      T_n+2T_{n+3}+T_{n+6}=0.    
    \end{equation}

 We now proceed by induction on $k\geq 0$ to prove \Cref{eq:T_n}. We only provide details for the $n=3k+1$ case and the remaining cases follow similarly. 
 
 Since $T_1=1$ and $T_4=-2$,  the claim holds for $k=0, 1$.  Suppose $k\geq 1$ and consider $T_n$ when $n=3(k+1)+1$. Our goal is to show $T_n= (-1)^{k+1}(k+2)$. By the induction hypothesis we have 
    $$T_{n-3}= T_{3k+1}= (-1)^k (k+1) \text{ and } T_{n-6}=T_{3(k-1)+1}=(-1)^{k-1}k.$$
    It then follows from Equation (\ref{eq:Tn_recursion}) that 
    $$T_n= -2T_{n-3}-T_{n-6}=(-1)^{k+1}2 (k+1)+ (-1)^k k= (-1)^{k+1}(k+2).$$
\end{proof}

\begin{thm}\label{thm:path_cycle}
The degree of the $h$-polynomial of $R/I(C_n)$ for $n\geq 3$ is 
$$  \deg h_{R/I(C_n)}(t)  = \left\lfloor \frac{n}{2} \right\rfloor.$$
The degree of the $h$-polynomial of $R/I(P_n)$ for $n\geq 1$ is 
$$ \deg h_{R/I(P_n)}(t) = \begin{cases}
         \left\lceil \frac{n}{2} \right\rceil & \text{if } n \equiv  0,2 \pmod 3,\\
        \left\lceil \frac{n}{2} \right\rceil-1 & \text{if } n \equiv 1 \pmod 3.
    \end{cases}$$ 
\end{thm}

\begin{proof}
    It follows from Lemmas \ref{Thm:alternate_sum} and \ref{Sumofgpath_cycle}  that the degree of the $h$-polynomial of $R/I(C_n)$ is $\alpha(C_n)= \left\lfloor \frac{n}{2} \right\rfloor$ for $n\geq 3$ and the degree of the $h$-polynomial of $R/I(P_n)$ is $ \alpha(P_n)= \left\lceil \frac{n}{2} \right\rceil$  when $n \equiv 0, 2 \pmod 3$ for $n\geq 1$. It remains to show $\deg h_{R/I(P_n)}(t) = \left\lceil \frac{n}{2} \right\rceil -1$ when $n \equiv 1 \pmod 3$. By Corollary  \ref{DegreeintermsofD}, this is equivalent to showing $D_0 \neq 0$  for the path $P_n$ when $n\equiv 1 \pmod 3$. This follows immediately from \Cref{D0Values}.
\end{proof}

\section{Degrees of Bipartite Graphs}\label{sec:bipartite}
In this section, we study the degrees of the $h$-polynomials of edge ideals of bipartite graphs. In particular, we determine which conditions guarantee the degree of the $h$-polynomial of a bipartite graph is equal to its independence number.

For the remainder of this section, let $G$ denote a connected bipartite graph with a bipartition $(U,V)$ such that $n=|U| \geq |V|=m$. 

As noted in \Cref{Thm:alternate_sum}, the degree of the \( h \)-polynomial of \( R/I(G) \) is determined by \( g = \sum_{j=1}^{\alpha} (-1)^{j-1}g_j\), where \( g_j \) is the number of independent sets of cardinality \( j \) in \( G \). Consequently, our focus is on counting these independent sets. To facilitate this, we partition the sets \( U \) and \( V \) according to the following notation, which will be used throughout this section.

\begin{notation}\label{not:vertex_partition}
   We use the notation $V_W$ to denote the set of all whiskered vertices of $G$ in $V$.  
      $$V_W=\{ v\in V : \{u,v\}\in E(G) \text{ for some leaf } u\in U  \}$$
    The complement of $V_W$ in $V$ is denoted by $\overline{V_W}$. 
    
    We use the notation  $U_L$ to denote a set of leaves in $U$ such that for each $v\in V_W$ there exists exactly one leaf $u\in U_L$ such that $\{u,v\} \in E(G)$ so that $|U_L|=|V_W|$. The complement of $U_L$ in $U$ is denoted by $\overline{U_L}$. Unlike $V_W$, there are multiple possible choices for the set $U_L$. However, the following combinatorial results do not depend on this choice. 
\end{notation}

\begin{ex}\label{ex:3}
Consider the following bipartite graph in \Cref{fig:notation} with the bipartition $(U,V)$ where $U=\{u_1,u_2,u_3, u_4, u_5, u_6\}$ and $V=\{v_1,v_2,v_3, v_4, v_5\}$. 
\begin{figure}[ht]
    \centering
    \begin{tikzpicture}[line cap=round,line join=round,>=triangle 45,x=1cm,y=1cm]
    \clip(-4.17354858383824,-0.49721795610347175) rectangle (4.143337795929728,4.292965128408889);
        \draw [line width=1pt] (-1.98,0)-- (-2,1.9);
        \draw [line width=1pt] (-1.98,0)-- (-1,1.9);
        \draw [line width=1pt] (-1.98,0)-- (0,1.9);
        \draw [line width=1pt] (0,1.9)-- (0,0);
        \draw [line width=1pt] (0,1.9)-- (2,0);
        \draw [line width=1pt] (2,0)-- (1,2);
        \draw [line width=1pt] (2,0)-- (2,2);
        \draw [line width=1pt] (-2,2.1)-- (-2,3.9);
        \draw [line width=1pt] (-1,2.1)-- (-1,3.9);
        \draw [line width=1pt] (0,2.1)-- (0,3.9);
    \begin{small}
        \draw [fill=black] (-1.98,0) circle (2.5pt);
        \draw[color=black] (-1.8,-0.25) node {$u_1$};
        \draw [color=black] (-2,2) ++(-3pt,0 pt) -- ++(3pt,3pt)--++(3pt,-3pt)--++(-3pt,-3pt)--++(-3pt,3pt);
        \draw[color=black] (-1.8,2.18) node {$v_1$};
        \draw [color=black] (-1,2) ++(-3pt,0 pt) -- ++(3pt,3pt)--++(3pt,-3pt)--++(-3pt,-3pt)--++(-3pt,3pt);
        \draw[color=black] (-0.8,2.18) node {$v_2$};
        \draw [color=black] (0,2) ++(-3pt,0 pt) -- ++(3pt,3pt)--++(3pt,-3pt)--++(-3pt,-3pt)--++(-3pt,3pt);
        \draw[color=black] (0.2,2.18) node {$v_3$};
        \draw [fill=black] (0,0) circle (2.5pt);
        \draw[color=black] (0.2,-0.25) node {$u_2$};
        \draw [fill=black] (2,0) circle (2.5pt);
        \draw[color=black] (2.2,-0.25) node {$u_3$};
        \draw [fill=black] (1,2) ++(-3pt,0 pt) -- ++(3pt,3pt)--++(3pt,-3pt)--++(-3pt,-3pt)--++(-3pt,3pt);
        \draw[color=black] (1.2,2.18) node {$v_4$};
        \draw [fill=black] (2,2) ++(-3pt,0 pt) -- ++(3pt,3pt)--++(3pt,-3pt)--++(-3pt,-3pt)--++(-3pt,3pt);
        \draw[color=black] (2.2,2.18) node {$v_5$};
        % Wisker Nodes % % % % % % % % % % % % % % % 
        \draw [color=black] (-2,4) circle (2.5pt);
        \draw[color=black] (-1.8,4.18) node {$u_4$};
        \draw [color=black] (-1,4) circle (2.5pt);
        \draw[color=black] (-0.8,4.18) node {$u_5$};
        \draw [color=black] (0,4) circle (2.5pt);
        \draw[color=black] (0.2,4.18) node {$u_6$};
    \end{small}
\end{tikzpicture}
    \caption{Sets $U_L , V_W, \overline{U_L}$ and $\overline{V_W}$}
    \label{fig:notation}
\end{figure}
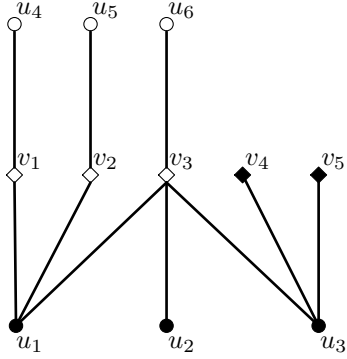
   
For this graph, we have  
$$V_W=\{v_1,v_2,v_3\}, U_L =\{u_4,u_5,u_6\}, \overline{U_L}=\{u_1,u_2,u_3\} \text{ and } \overline{V_W}=\{v_4,v_5\}.$$
\end{ex}

Building on \Cref{not:vertex_partition}, our goal of counting the independent sets in \(G\) can be reduced to analyzing which vertices of \(\overline{U_L}\) are included in an independent set. For instance, the number of independent sets of size \(j\) that do not contain any vertex of \(\overline{U_L}\) is
$$\sum_{i=0}^{j}  2^i \binom{|U_L|}{i}\binom{|\overline{V_W}|}{j-i}.$$
In what follows, we provide a similar formula for any subset of \(\overline{U_L}\) and use it to express \(g\) in terms of these subsets. We start with the following notation, which is used for the remainder of this section.

\begin{notation}
  Let $S$ denote a subset of $\overline{U_L}$ for the remainder of this section.  
  \begin{enumerate}
      \item We use the notation $g_j (S)$ to denote the number of independent sets in $G$ of size $j$ whose intersection with $\overline{U_L}$ is exactly $S$. 
  $$g_j (S) = |\{ \mathcal{C} \text{ is an independent set of } G : |\mathcal{C}|=j \text{ and } \mathcal{C} \cap \overline{U_L} = S  \}|$$
The notation $g_j(S)$ is meant to be evocative of $g_j(G)$, but it should not be confused with it,  $g_j(G)=g_j$ is the number of independent subsets of size $j$ for the graph $G$.

  Similar to before, we use $g(S)$ to denote the alternating sum 
    \[
        g(S)=\sum_{j=1}^{\alpha(G)}(-1)^{j-1}g_j(S).
    \]  
  \item Let $N(S)$ denote the set of all neighbors of vertices in $S$ in $G$. Note that $N(S)\subseteq V$. We introduce the following two sets, $W_S$ and $W'_S$ which are the sets of elements in $V_W$ that are not in $N(S)$ and elements in $\overline{V_W}$ that are not in $N(S)$, respectively:
    \begin{align*}
        W_S &  = V_W \setminus N(S)\\
        W_S' &= \overline{V_W} \setminus N(S).
    \end{align*}
  \end{enumerate}     
\end{notation}

This notation is illustrated in the following example.

   \begin{ex}
    For the bipartite graph in Figure \ref{fig:notation}, let $S_1=\{u_2,u_3\}, S_2= \{u_1,u_3\}$ and $S_3=\{u_1,u_2\}$. 

    To construct an independent set of size three whose intersection with $\overline{U_L}$ is exactly $S_1$, we can add any one of the vertices $v_1,v_2,u_4,u_5,$ or $u_6$. Thus $g_3(S_1)=5$. Similarly $g_3(S_2)=3$ %\Gabe{I think $g_3(S_2)=3$, the third vertex could be any of $u_4$, $u_5$ and $u_6$} 
    while $g_3(S_3)=4$.
    %\Gabe{I think $g_3(S_3)=5$ the third vertex could be any of $u_4$, $u_5$, $u_6$, $v_4$ and $v_5$}.

    We also see that 
    \begin{align*}
        N(S_1) &=\{v_3,v_4,v_5\}, W_{S_1}= \{v_1,v_2\} \text{ and } W'_{S_1}= \emptyset,\\
       N(S_2) &=\{v_1,v_2,v_3,v_4,v_5\},  W_{S_2} = \emptyset \text{ and } W'_{S_2}= \emptyset,\\
       N(S_3)&=\{v_1,v_2,v_3\},  W_{S_3} = \emptyset \text{ and } W'_{S_3}= \{v_4,v_5\}.
    \end{align*}
\end{ex}

   \begin{remark}\label{rmk:g(S)}
   Since independent sets in $G$ can be counted in terms of how many vertices of $\overline{U_L}$ are contained in them,  one can express the number of independent sets of size $j$ as 
   $$g_j=\displaystyle \sum_{S \subseteq \overline{U_L}} g_j (S).$$
   \begin{enumerate}
       \item[(a)]  With the above expression of $g_j$, we have
    \begin{equation}\label{eq:g(S)}
     \displaystyle g=\sum_{j=1}^{\alpha} (-1)^{j-1}\left( \sum_{S \subseteq \overline{U_L}} g_j (S)\right)=\sum_{S \subseteq \overline{U_L}} \left( \sum_{j=1}^{\alpha} (-1)^{j-1} g_j (S) \right)= \sum_{S \subseteq \overline{U_L}} g(S)    
    \end{equation} 
    where $\displaystyle g(S):=  \sum_{j=1}^{\alpha} (-1)^{j-1} g_j (S)$.
    \item[(b)]  When  $S\neq \emptyset$, we can express $g(S)$ as follows: 
 \[g(S)=\sum_{j=1}^{\alpha} (-1)^{j-1} g_j (S)= \sum_{j=\vert S\vert}^{\alpha} (-1)^{j-1}g_j (S)=(-1)^{\vert{S}\vert}\sum_{j=0}^{\alpha -\vert S \vert} (-1)^{j-1} g_{j+\vert S\vert} (S) \] 
such that
 \[g_{j+\vert S\vert} (S)=\sum_{i=0}^{j} 2^i \binom{| W_S|}{i}\binom{|W'_S|+| U_L |-| W_S|}{j-i}. \]
 The last expression for $g_{j+|S|}(S)$ represents how many ways one can construct an independent set of size $j+|S|$ whose intersection with $\overline{U_L}$ is $S$ by choosing the remaining $j$ vertices. In particular, we can first choose $i$ vertices from either $|W_S|$ or $N(W_S)\cap U_L$ and then choose $j-i$ vertices from the remaining vertices in $G$ that are not adjacent to $S$ and not in $\overline{U_L}$. 
\item[(c)] As previously discussed,  when $S=\emptyset$, we have $$\displaystyle g(\emptyset) =\sum_{j=1}^{m} (-1)^{j-1} g_{j} (\emptyset) \text{ such that } g_{j} (\emptyset)=\sum_{i=0}^{j}  2^i \binom{|U_L|}{i}\binom{| \overline{V_W}|}{j-i}. $$ 
   \end{enumerate}
   \end{remark}

Equation (\ref{eq:g(S)}) suggests that value of $g$ can be analyzed by studying \( g(S) \) for \( S \subseteq \overline{U_L} \). Thus we investigate \( g(S) \) by considering which vertices of $V$ are whiskered and which are not.  In particular, we examine the cases \( \overline{V_W} = \emptyset \) and  \( \overline{V_W} \neq \emptyset \). 
 
    \begin{lem}\label{lem:empty}
Suppose $\overline{V_W}=\emptyset$, i.e. $V_W=V$.  Then $g(S)=0$ when $S\neq \emptyset$; otherwise, 
\[ g(\emptyset)= \begin{cases}
    0 &  \text{if $|V|$ is even},\\
    2 &   \text{if $|V|$ is odd}.
\end{cases}
\]
    \end{lem}

    \begin{proof}
Under the assumption of this lemma, we have $m=|V_W|=|U_L|$. In addition,     $W'_S=\emptyset$ for any $S \subseteq \overline{U_L}$.  Also note that in this case $\alpha(G) = n$.

Suppose $S\neq \emptyset$. Observe that $W_S \subsetneq V_W$ since $N(S)\neq \emptyset$ as $G$ is a connected bipartite graph and $\overline{V_W}$ is assumed to be empty. Next, by \Cref{rmk:g(S)} part (b) and reversing the order of the double sum, we obtain the following set of equalities:
    \begin{align*}
        g(S) &= (-1)^{\vert{S}\vert}\sum_{j=0}^{n-\vert S \vert} (-1)^{j-1} g_{j+\vert S\vert} (S)\\
        &=(-1)^{\vert{S}\vert+1}\sum_{j=0}^{n-\vert S \vert} \left( \sum_{i=0}^{j} (-1)^{j} 2^i \binom{| W_S|}{i}\binom{m-|W_S|}{j-i} \right)\\
        &= (-1)^{\vert{S}\vert+1}\sum_{i=0}^{n-|S|} \left( \sum_{j=i}^{n-| S|} (-1)^j 2^i \binom{| W_S|}{i}\binom{m-| W_S|}{j-i} \right)\\
        &= (-1)^{\vert{S}\vert+1} \sum_{i=0}^{n-\vert S \vert} (-2)^i\binom{|W_S|}{i} \left(\sum_{j=i}^{n-\vert S \vert} (-1)^{j-i}\binom{m-|W_S|}{j-i}\right)\\
        &= (-1)^{\vert{S}\vert+1} \sum_{i=0}^{\vert W_S \vert} (-2)^i\binom{|W_S|}{i} \left(\sum_{k=0}^{m-|W_S|} (-1)^{k}\binom{m-|W_S|}{k}\right)
    \end{align*}
where the last equality follows from substituting $k=j-i$ in the inner sum and the following facts: (1) $n-\vert S\vert \geq n-\left\vert \overline{U_L} \right\vert=\vert U_L \vert =\vert V_W \vert=\vert V\vert =m $, (2) $n-\vert S\vert -i\geq n-\vert S \vert -\vert W_S\vert\geq m- \vert W_S \vert $, for all $0 \leq i \leq \vert W_S\vert$, and (3) $\vert V_W \vert = \vert V \vert > \vert W_S\vert$.

By the Binomial Theorem and the fact that $W_S\subsetneq V_W$, we conclude that $g(S)=0$ when $S\neq \emptyset$ as follows:
$$g(S)= (-1)^{\vert{S}\vert+1} \sum_{i=0}^{\vert W_S \vert} (-2)^i\binom{|W_S|}{i} \underbrace{\left( (1-1)^{m-|W_S|} \right)}_{0}=0.$$
% Therefore, $g(S)=0$ for any $S\neq \emptyset$, as desired. 

Suppose $S=\emptyset$. Recall from \Cref{rmk:g(S)} part (b) that
$$ g(\emptyset) = \sum_{j=1}^{m} (-1)^{j-1} g_{j} (\emptyset) =\sum_{j=1}^m(-1)^{j-1}\left(\sum_{i=0}^{j} 2^i \binom{m}{i}\binom{|\overline{V_W} |}{j-i} \right) = \sum_{j=1}^m(-1)^{j-1} 2^j \binom{m}{j}.$$
where the last equality follows from the fact that  $\displaystyle \binom{|\overline{V_W} |}{j-i}=0$  unless $j=i$ as $|\overline{V_W} |=0.$  Then, by the Binomial Theorem, we obtain the desired equality as follows:
$$g(\emptyset) =1-\sum_{j=0}^m  (-2)^j \binom{m}{j}= 1-(1-2)^m =\begin{cases} 
0 & \text{ if $m= |V|$ is even}, \\ 
2 & \text{ if $m=|V|$ is odd}.    
\end{cases}$$  
  \end{proof}

\begin{lem}\label{lem:nonempty}
   Suppose $\overline{V_W} \neq\emptyset$. Then,  $g(\emptyset)=1$ and for $S\neq \emptyset$
   $$g(S)=
   \begin{cases} (-1)^{|S |+ |V_W|+1} & \text{if }N(S)=\overline{V_W},  \\ 0 & \text{otherwise}. \end{cases}$$
\end{lem}

    \begin{proof}
Recall from \Cref{rmk:g(S)} part (b) that 
$$ g(\emptyset) =\sum_{j=1}^{m} (-1)^{j-1} g_{j} (\emptyset) =\sum_{j=1}^m(-1)^{j-1}\left(\sum_{i=0}^{j} 2^i\binom{|U_L |}{i}\binom{|\overline{V_W} |}{j-i} \right).$$
By separating out the $i=0$ term and then re-indexing the first summation, we can express  $g(\emptyset)$ as follows:
\begin{equation}\label{eq:g(empty)}
g(\emptyset)  =1- \sum_{j=0}^{m} (-1)^{j}\binom{|\overline{V_W} |}{j} -\sum_{j=1}^m\left(\sum_{i=1}^{j}(-1)^{j} 2^i\binom{|U_L |}{i}\binom{|\overline{V_W} |}{j-i} \right).
\end{equation}
 Note that the bounds of the second term of (\ref{eq:g(empty)}) can be restricted since $0<|\overline{V_W}| < m$. Thus, we can apply the Binomial Theorem to obtain
    \[
        \sum_{j=0}^{m} (-1)^{j}\binom{|\overline{V_W} |}{j}= \sum_{j=0}^{|V_W|} (-1)^{j}\binom{|\overline{V_W} |}{j}=(1-1)^{|\overline{V_W}|} =0.
    \]
 In addition, by reversing the order of the sum in the third term of \Cref{eq:g(empty)} and substituting $k=j-i$, we obtain 
    \begin{align*}
        \sum_{i=1}^m\left(\sum_{j=i}^{m}(-1)^{j} 2^i\binom{|U_L |}{i}\binom{|\overline{V_W} |}{j-i} \right)   
            &= \sum_{i=1}^m (-2)^i\binom{|U_L |}{i}\left(\sum_{j=i}^{m} (-1)^{j-i}\binom{|\overline{V_W} |}{j-i} \right)\\
            &= \sum_{i=1}^m(-2)^i\binom{|U_L|}{i}\left(\sum_{k=0}^{m-i}(-1)^k\binom{|\overline{V_W}|}{k}\right).
    \end{align*}
% Since $|U_L|=|V_W|<m$ and $|\overline{V_W}|=m-|U_L|$, we can restrict the bounds of the last term and again use the Binomial Theorem to conclude that the last terms is zero, thus, $g(\emptyset)=1$.

Since $|U_L| = |V_W| < m$ and $|\overline{V_W}| = m - |U_L|$, we can restrict the bounds of the above sum and, by applying the Binomial Theorem, conclude that the last term of \Cref{eq:g(empty)} is zero. Therefore, $g(\emptyset) = 1$.

For $S\neq \emptyset$, let $p =|W_S'|+|U_L|-|W_S|$. Then, using \Cref{rmk:g(S)} part (b) and following the same steps from the proof of \Cref{lem:empty}, we can express $g(S)$ as:
$$g(S)= (-1)^{|S|+1} \sum_{i=0}^{\alpha-|S|} (-2)^i\binom{|W_S|}{i} \left(\sum_{k=0}^{\alpha-|S|-i} (-1)^{k}\binom{p}{k}\right).$$
 Note that disjoint union $W_S\sqcup S$ is an independent set in $G$, as is $W_S'\sqcup U_L \sqcup S$. Therefore $|W_S|\leq \alpha-|S|$ and  $|W_S'|+|U_L|\leq \alpha-|S|$. The latter implies that $p \leq \alpha -|S|-i$ whenever $i\leq |W_S|$. Thus, by restricting the bounds of the sums, we get
$$
   g(S) = (-1)^{|S|+1} \sum_{i=0}^{|W_S|} (-2)^i\binom{|W_S|}{i} \left(\sum_{k=0}^{p} (-1)^{k}\binom{p}{k}\right).
$$
It is immediate that  the inner sum is equal to $1$ when $p=0$. Otherwise, it is equal to zero by the Binomial Theorem. Thus, $g(S)=0$ when $p\neq 0$.

Finally, observe that  $p=0$  if and only if $|W_S'|=0$ and $W_S=V_W$ since $W_S \subseteq V_W$ and $|U_L|=|V_W|$. This means $p=0$ if and only if $N(S)=\overline{V_W}.$  In this case, we have
$$
   g(S) = (-1)^{|S|+1} \sum_{i=0}^{|V_W|} (-2)^i\binom{|V_W|}{i} =(-1)^{|S|+|V_W|+1}.$$
Therefore, we obtain the desired equality.
   \end{proof}

Before the main result of this section, we introduce notation to simplify the statement of our result.

\begin{notation}
  Let \( X \) denote the collection of non-empty subsets of \( \overline{U_L} \) such that the neighborhood of each element in this collection is \( \overline{V_W} \), i.e.,
$$
X = \{ S \subseteq \overline{U_L} : S \neq \emptyset \text{ and } N(S) = \overline{V_W} \}.
$$
In addition, let $q$ denote the difference between the even cardinality elements  and odd cardinality elements in $X$, i.e.
$$q=|\{S\in X : |S| \text{ is even}\}|-|\{T \in X : |T| \text{ is odd}\}|.$$  
\end{notation}

Now, we present the main result of this section.

\begin{thm}\label{thm:bipartite}
 Let $G$ be a connected bipartite graph on vertex set $U\cup V$. 
\begin{enumerate}[a]
        \item If every vertex in $V$ is whiskered so that $\overline{V_W}=\emptyset$, then  $\deg h_{R/I(G)} (t)= \alpha(G)$.
        \item If there are unwhiskered vertices in $V$, $\overline{V_W}\neq \emptyset$, then  $\deg h_{R/I(G)} (t)= \alpha(G)$ if and only if $q\neq 0$.
    \end{enumerate}
\end{thm}

\begin{proof}
   Suppose $\overline{V_W} = \emptyset$. Then by Equation (\ref{eq:g(S)}) and Lemma  \ref{lem:empty} we see that
$$g= \sum_{j=1}^{n} (-1)^{j-1} g_j= \sum_{S\subseteq \overline{U_L}} g(S)= \begin{cases}
    0 &: |V| \text{ is even},\\
    2 &: |V| \text{ is odd}
\end{cases}$$
so that $g\neq 1$. 
Thus,  $\deg h_{R/I(G)} (t)= \alpha (G)$ by \Cref{Thm:alternate_sum} which proves part (a).

Now suppose \( \overline{V_W} \neq \emptyset \). Recall from  \Cref{lem:nonempty} that \( g(S) = (-1)^{|S| + |V_W| + 1} \) whenever \( S \in X \) and it is zero otherwise. So, using Equation (\ref{eq:g(S)}) and \Cref{lem:nonempty}, we have

$$g = \sum_{S\subseteq \overline{U_L}} g(S)
     = 1+ \sum_{\substack{S \subseteq \overline{U_L}\\
     S \neq \emptyset}} g(S)
     % &= 1+ (-1)^{|V_W|+1}\sum_{S\in X} (-1)^{|S|}
    = 1- (-1)^{|V_W|}\sum_{S\in X} (-1)^{|S|}
    = 1- (-1)^{|V_W|} q.$$
Then, by  \Cref{Thm:alternate_sum},  $\deg h_{R/I(G)} (t)= \alpha(G)$ if and only if $q\neq 0$,  proving part (b).
\end{proof}

Observe that \( q = -1 \) when \( G \) is a complete bipartite graph $K_{n,m}$ with $n\geq m >1$ because \( N(S) = \overline{V_W} = V \) for any \( S \subseteq \overline{U_L} = U \) and \( X \cup \{\emptyset\} = \mathcal{P}(U) \). Additionally, for a complete bipartite graph, the number of even and odd cardinality subsets of \( U \) (including the empty set) are equal. In the following corollary, we present a more general version of this observation.  

 \begin{cor}\label{SpecificCases}
Let $G$ be a connected bipartite graph.
\begin{enumerate}[a]
\item If every vertex in $V$ is whiskered, then $\deg h_{R/I(G)} (t)= \alpha(G) = n$.
\item If there are vertices in $V$ that are not whiskered and  $X\cup \{\emptyset\}= \mathcal{P}(T)$ for some non-empty subset $T\subseteq \overline{U_L} $, then $\deg h_{R/I(G)} (t)=\alpha(G)$. In particular, if $G$ is a complete bipartite graph, then $\deg h_{R/I(G)} (t)= n$. 
\end{enumerate}       
   \end{cor}

We conclude this section with a couple of examples and two questions that we hope will attract the attention of interested researchers.

\begin{ex}
    Let $H$ be the following connected bipartite graph in \Cref{fig:ex_2}:
    \begin{figure}[ht]
    \centering
    \begin{tikzpicture}[line cap=round,line join=round,>=triangle 45,x=1cm,y=1cm]
    \clip(-4.17354858383824,0.19721795610347175) rectangle (4.143337795929728,3.692965128408889);
        \draw [line width=1pt] (-1.98,0.7)-- (-2,1.9);
        \draw [line width=1pt] (-1.98,0.7)-- (-1,1.9);
        \draw [line width=1pt] (-1.98,0.7)-- (0,1.9);
        \draw [line width=1pt] (0,1.9)-- (0,0.7);
        \draw [line width=1pt] (1,1.9)-- (0,0.7);
        % \draw [line width=1pt] (0,2)-- (2,0);
        \draw [line width=1pt] (2,0.7)-- (1,2);
        \draw [line width=1pt] (2,0.7)-- (2,2);
        \draw [line width=1pt] (-2,2.1)-- (-2,3.1);
        \draw [line width=1pt] (-1,2.1)-- (-1,3.1);
        \draw [line width=1pt] (0,2.1)-- (0,3.1);
        \draw [line width=1pt] (0,2.1)-- (0,3.1);
    \begin{small}
        \draw [fill=black] (-1.98,0.7) circle (2.5pt);
        \draw[color=black] (-1.8,0.45) node {$u_1$};
        \draw [color=black] (-2,2) ++(-3pt,0 pt) -- ++(3pt,3pt)--++(3pt,-3pt)--++(-3pt,-3pt)--++(-3pt,3pt);
        \draw[color=black] (-1.8,2.18) node {$v_1$};
        \draw [color=black] (-1,2) ++(-3pt,0 pt) -- ++(3pt,3pt)--++(3pt,-3pt)--++(-3pt,-3pt)--++(-3pt,3pt);
        \draw[color=black] (-0.8,2.18) node {$v_2$};
        \draw [color=black] (0,2) ++(-3pt,0 pt) -- ++(3pt,3pt)--++(3pt,-3pt)--++(-3pt,-3pt)--++(-3pt,3pt);
        \draw[color=black] (0.2,2.18) node {$v_3$};
        \draw [fill=black] (0,0.7) circle (2.5pt);
        \draw[color=black] (0.2,0.45) node {$u_2$};
        \draw [fill=black] (2,0.7) circle (2.5pt);
        \draw[color=black] (2.2,0.45) node {$u_3$};
        \draw [fill=black] (1,2) ++(-3pt,0 pt) -- ++(3pt,3pt)--++(3pt,-3pt)--++(-3pt,-3pt)--++(-3pt,3pt);
        \draw[color=black] (1.2,2.18) node {$v_4$};
        \draw [fill=black] (2,2) ++(-3pt,0 pt) -- ++(3pt,3pt)--++(3pt,-3pt)--++(-3pt,-3pt)--++(-3pt,3pt);
        \draw[color=black] (2.2,2.18) node {$v_5$};
        % Wisker Nodes % % % % % % % % % % % % % % % 
        \draw [color=black] (-2,3.2) circle (2.5pt);
        \draw[color=black] (-1.8,3.38) node {$u_4$};
        \draw [color=black] (-1,3.2) circle (2.5pt);
        \draw[color=black] (-0.8,3.38) node {$u_5$};
        \draw [color=black] (0,3.2) circle (2.5pt);
        \draw[color=black] (0.2,3.38) node {$u_6$};
    \end{small}
\end{tikzpicture}
    \caption{The bipartite graph $H$}
    \label{fig:ex_2}
\end{figure}

The sets $U, V, U_L, V_W, \overline{U_L}$ and $\overline{V_W}$ of $H$ are the same as the ones for $G$ in \Cref{ex:3}. Notice that for this graph, $S=\{u_3\}$ is the only non-empty subset of $\overline{U_L}$ such that $N(S)=\{v_4,v_5\}=\overline{V_W}$. So, $X=\{\{u_3\}\}$ and $q=-1$ for this graph $H$. Thus,   $\deg h_{R/I(H)} (t)=\alpha (H)= 6$ by \Cref{thm:bipartite}.
\end{ex}

\begin{ex}
    Let $H'$ be the following connected bipartite graph in \Cref{fig:ex_3}:
    \begin{figure}[ht]
    \centering
    \begin{tikzpicture}[line cap=round,line join=round,>=triangle 45,x=1cm,y=1cm]
    \clip(-4.17354858383824,-1.49721795610347175) rectangle (4.143337795929728,1.692965128408889);
        \draw [line width=1pt] (-3,1.1)-- (-3,0.09);
        \draw [line width=1pt] (-3,-0.09)-- (-3,-1.2);
        \draw [line width=1pt] (-3,-1.2)-- (-1,0);
        \draw [line width=1pt] (-1,0)-- (-1,-1.2);
        \draw [line width=1pt] (1,-1.2)-- (-1,0);
        \draw [line width=1pt] (-1,-1.2)-- (3,0);
        \draw [line width=1pt] (1,-1.2)-- (3,0);
        \draw [line width=1pt] (3,0)-- (3,-1.2);
        \draw [line width=1pt] (3,-1.2)-- (-1,0);
        \draw [line width=1pt] (1,-1.2)-- (1,0);
        %\draw [line width=1pt] (0,2.1)-- (0,3.9);
       %\draw [line width=1pt] (0,2.1)-- (0,3.9);
    \begin{small}
       
        \draw [color=black] (-3,0) ++(-3pt,0 pt) -- ++(3pt,3pt)--++(3pt,-3pt)--++(-3pt,-3pt)--++(-3pt,3pt);
        \draw[color=black] (-2.7,0.1) node {$v_1$};
        \draw [fill=black] (-1,0) ++(-3pt,0 pt) -- ++(3pt,3pt)--++(3pt,-3pt)--++(-3pt,-3pt)--++(-3pt,3pt);
        \draw[color=black] (-0.7,0.1) node {$v_2$};
        \draw [fill=black] (1,0) ++(-3pt,0 pt) -- ++(3pt,3pt)--++(3pt,-3pt)--++(-3pt,-3pt)--++(-3pt,3pt);
        \draw[color=black] (1.3,0.1) node {$v_3$};
        \draw [fill=black] (3,0) ++(-3pt,0 pt) -- ++(3pt,3pt)--++(3pt,-3pt)--++(-3pt,-3pt)--++(-3pt,3pt);
        \draw[color=black] (3.3,0.1) node {$v_4$};
        \draw [fill=black] (-3,-1.2) circle (2.5pt);
        \draw[color=black] (-2.7,-1.38) node {$u_1$};
        \draw [fill=black] (-1,-1.2) circle (2.5pt);
        \draw[color=black] (-0.7,-1.38) node {$u_2$};
        \draw [fill=black] (1,-1.2) circle (2.5pt);
        \draw[color=black] (1.3,-1.38) node {$u_3$};
         \draw [fill=black] (3,-1.2) circle (2.5pt);
        \draw[color=black] (3.3,-1.38) node {$u_4$};
        %\draw [fill=black] (1,2) ++(-3pt,0 pt) -- ++(3pt,3pt)--++(3pt,-3pt)--++(-3pt,-3pt)--++(-3pt,3pt);
        %\draw[color=black] (1.2,2.18) node {$v_4$};
       % \draw [fill=black] (2,2) ++(-3pt,0 pt) -- ++(3pt,3pt)--++(3pt,-3pt)--++(-3pt,-3pt)--++(-3pt,3pt);
        %\draw[color=black] (2.2,2.18) node {$v_5$};
        % Wisker Nodes % % % % % % % % % % % % % % % 
        \draw [color=black] (-3,1.2) circle (2.5pt);
        \draw[color=black] (-2.8,1.38) node {$u_5$};
        %\draw [color=black] (-1,4) circle (2.5pt);
        %\draw[color=black] (-0.8,4.18) node {$u_5$};
        %\draw [color=black] (0,4) circle (2.5pt);
        %\draw[color=black] (0.2,4.18) node {$u_6$};
    \end{small}
\end{tikzpicture}
    \caption{The bipartite graph $H'$}
    \label{fig:ex_3}
\end{figure}

In this situation $U=\{u_1, u_2,u_3, u_4, u_5\}$, $V=\{v_1,v_2,v_3,v_4\}$,  $V_W=\{v_1\}$, $U_L=\{u_5\}$, $ \overline{U_L}=\{u_1,u_2,u_3,u_4\}$ and $\overline{V_W}=\{v_2,v_3,v_4\}$. Notice that $X=\{\{u_3\},\{u_2,u_3\},\{u_3,u_4\},\{u_2,u_3,u_4\}\}$ and $q=0$ for this graph $H$. Thus,   $\deg h_{R/I(H)} (t)<\alpha (H)= 5$ by \Cref{thm:bipartite}. In fact, $\deg h_{R/I(H)} (t)=3$.
\end{ex}

\begin{remark}
    There are many situations in which $q\neq 0$, besides the one mentioned in \Cref{SpecificCases}(b). Two easy to check circumstances in which this happens are: 
    \begin{itemize}
        \item[(1)] if $X=\left\{T': T'\subseteq T \text{ and } \vert T'\vert \geq k \right\}$ for some $T \subseteq \overline{U_L}$ and some positive integer $k$ (see \Cref{fig:ex_4} for an example of such a bipartite graph),
    \end{itemize}
 \begin{figure}[H]
    \centering
    \begin{tikzpicture}[line cap=round,line join=round,>=triangle 45,x=1cm,y=1cm]
    \clip(-4.17354858383824,-1.49721795610347175) rectangle (4.143337795929728,1.692965128408889);
        \draw [line width=1pt] (-3,1.1)-- (-3,0.09);
        \draw [line width=1pt] (-3,-0.09)-- (-3,-1.2);
        \draw [line width=1pt] (-3,-1.2)-- (-1,0);
        \draw [line width=1pt] (-1,0)-- (-1,-1.2);
        \draw [line width=1pt] (1,-1.2)-- (-1,0);
        \draw [line width=1pt] (-1,-1.2)-- (1,0);
        \draw [line width=1pt] (1,-1.2)-- (3,0);
        \draw [line width=1pt] (3,0)-- (3,-1.2);
        \draw [line width=1pt] (3,-1.2)-- (1,0);
        \draw [line width=1pt] (-3,-1.2)-- (-1,0);
       \draw [line width=1pt] (-3,-1.2)-- (1,0);
       \draw [line width=1pt] (-3,-1.2)-- (3,0);
       
    \begin{small}
       
        \draw [color=black] (-3,0) ++(-3pt,0 pt) -- ++(3pt,3pt)--++(3pt,-3pt)--++(-3pt,-3pt)--++(-3pt,3pt);
        \draw[color=black] (-2.7,0) node {$v_1$};
        \draw [fill=black] (-1,0.1) ++(-3pt,0 pt) -- ++(3pt,3pt)--++(3pt,-3pt)--++(-3pt,-3pt)--++(-3pt,3pt);
        \draw[color=black] (-0.7,0.1) node {$v_2$};
        \draw [fill=black] (1,0) ++(-3pt,0 pt) -- ++(3pt,3pt)--++(3pt,-3pt)--++(-3pt,-3pt)--++(-3pt,3pt);
        \draw[color=black] (1.3,0.1) node {$v_3$};
        \draw [fill=black] (3,0) ++(-3pt,0 pt) -- ++(3pt,3pt)--++(3pt,-3pt)--++(-3pt,-3pt)--++(-3pt,3pt);
        \draw[color=black] (3.3,0.1) node {$v_4$};
        \draw [fill=black] (-3,-1.2) circle (2.5pt);
        \draw[color=black] (-2.7,-1.38) node {$u_1$};
        \draw [fill=black] (-1,-1.2) circle (2.5pt);
        \draw[color=black] (-0.7,-1.38) node {$u_2$};
        \draw [fill=black] (1,-1.2) circle (2.5pt);
        \draw[color=black] (1.3,-1.38) node {$u_3$};
         \draw [fill=black] (3,-1.2) circle (2.5pt);
        \draw[color=black] (3.3,-1.38) node {$u_4$};
        %\draw [fill=black] (1,2) ++(-3pt,0 pt) -- ++(3pt,3pt)--++(3pt,-3pt)--++(-3pt,-3pt)--++(-3pt,3pt);
        %\draw[color=black] (1.2,2.18) node {$v_4$};
       % \draw [fill=black] (2,2) ++(-3pt,0 pt) -- ++(3pt,3pt)--++(3pt,-3pt)--++(-3pt,-3pt)--++(-3pt,3pt);
        %\draw[color=black] (2.2,2.18) node {$v_5$};
        % Wisker Nodes % % % % % % % % % % % % % % % 
        \draw [color=black] (-3,1.2) circle (2.5pt);
        \draw[color=black] (-2.8,1.38) node {$u_5$};
        %\draw [color=black] (-1,4) circle (2.5pt);
        %\draw[color=black] (-0.8,4.18) node {$u_5$};
        %\draw [color=black] (0,4) circle (2.5pt);
        %\draw[color=black] (0.2,4.18) node {$u_6$};
    \end{small}
\end{tikzpicture}
    \caption{A bipartite graph with $X=\left\{T': T'\subseteq \{u_2,u_3,u_4\}  \text{ and } \vert T'\vert \geq 2 \right\}$.}
    \label{fig:ex_4}
\end{figure}
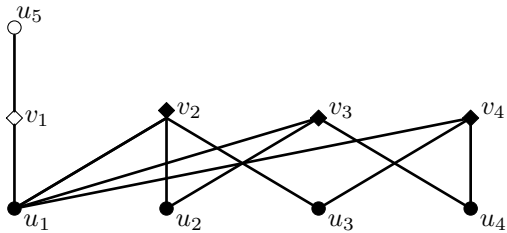

\begin{itemize}
    \item[(2)] if $X=\mathcal{P}(T_1) \setminus \left(\mathcal{P}(T_2) \cup \mathcal{P}(T_3) \right)$ for some $T_1,T_2, T_3 \subseteq \overline{U_L}$, with $T_2$ and $T_3$ disjoint proper subsets of $T_1$ (see \Cref{fig:ex_5} for an example of such a bipartite graph). 
\end{itemize}

\begin{figure}[H]
    \centering
    \begin{tikzpicture}[line cap=round,line join=round,>=triangle 45,x=1cm,y=1cm]
    \clip(-4.17354858383824,-1.49721795610347175) rectangle (6.143337795929728,1.692965128408889);
        \draw [line width=1pt] (-3,1.1)-- (-3,0.09);
        \draw [line width=1pt] (-3,-0.09)-- (-3,-1.2);
        \draw [line width=1pt] (-3,-1.2)-- (3,0);
        \draw [line width=1pt] (-1,0)-- (-1,-1.2);
        \draw [line width=1pt] (3,-1.2)-- (-1,0);
        \draw [line width=1pt] (1,-1.2)-- (1,0);
        \draw [line width=1pt] (3,-1.2)-- (1,0);
          \draw [line width=1pt] (3,0)-- (3,-1.2);
        \draw [line width=1pt] (5,0)-- (-1,-1.2);
        \draw [line width=1pt] (1,-1.2)-- (5,0);
        \draw [line width=1pt] (5,-1.2)-- (5,0);
       \draw [line width=1pt] (5,-1.2)-- (-2.9,0);
    \begin{small}
       
        \draw [color=black] (-3,0) ++(-3pt,0 pt) -- ++(3pt,3pt)--++(3pt,-3pt)--++(-3pt,-3pt)--++(-3pt,3pt);
        \draw[color=black] (-2.7,0.2) node {$v_1$};
        \draw [fill=black] (-1,0) ++(-3pt,0 pt) -- ++(3pt,3pt)--++(3pt,-3pt)--++(-3pt,-3pt)--++(-3pt,3pt);
        \draw[color=black] (-0.7,0.2) node {$v_2$};
        \draw [fill=black] (1,0) ++(-3pt,0 pt) -- ++(3pt,3pt)--++(3pt,-3pt)--++(-3pt,-3pt)--++(-3pt,3pt);
        \draw[color=black] (1.3,0.2) node {$v_3$};
        \draw [fill=black] (3,0) ++(-3pt,0 pt) -- ++(3pt,3pt)--++(3pt,-3pt)--++(-3pt,-3pt)--++(-3pt,3pt);
        \draw[color=black] (3.3,0.2) node {$v_4$};
         \draw [fill=black] (5,0) ++(-3pt,0 pt) -- ++(3pt,3pt)--++(3pt,-3pt)--++(-3pt,-3pt)--++(-3pt,3pt);
        \draw[color=black] (5.3,0.2) node {$v_5$};
        \draw [fill=black] (-3,-1.2) circle (2.5pt);
        \draw[color=black] (-2.7,-1.38) node {$u_1$};
        \draw [fill=black] (-1,-1.2) circle (2.5pt);
        \draw[color=black] (-0.7,-1.38) node {$u_2$};
        \draw [fill=black] (1,-1.2) circle (2.5pt);
        \draw[color=black] (1.3,-1.38) node {$u_3$};
         \draw [fill=black] (3,-1.2) circle (2.5pt);
        \draw[color=black] (3.3,-1.38) node {$u_4$};
        \draw [fill=black] (5,-1.2) circle (2.5pt);
        \draw[color=black] (5.3,-1.38) node {$u_5$};
        %\draw [fill=black] (1,2) ++(-3pt,0 pt) -- ++(3pt,3pt)--++(3pt,-3pt)--++(-3pt,-3pt)--++(-3pt,3pt);
        %\draw[color=black] (1.2,2.18) node {$v_4$};
       % \draw [fill=black] (2,2) ++(-3pt,0 pt) -- ++(3pt,3pt)--++(3pt,-3pt)--++(-3pt,-3pt)--++(-3pt,3pt);
        %\draw[color=black] (2.2,2.18) node {$v_5$};
        % Wisker Nodes % % % % % % % % % % % % % % % 
        \draw [color=black] (-3,1.2) circle (2.5pt);
        \draw[color=black] (-2.8,1.38) node {$u_5$};
       
        %\draw [color=black] (0,4) circle (2.5pt);
        %\draw[color=black] (0.2,4.18) node {$u_6$};
    \end{small}
\end{tikzpicture}
    \caption{A bipartite graph with $T_1=\{u_2,u_3,u_4\}$, $T_2=\{u_2,u_3\}$ and $T_3=\{u_4\}$.}
    \label{fig:ex_5}
\end{figure}
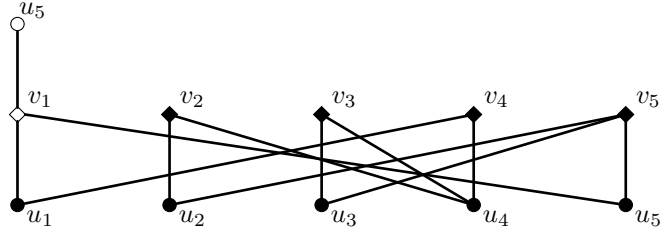
\end{remark}

Our first question is about identifying the classes of connected bipartite graphs for which \(\deg h_{R/I(G)}(t) = \alpha(G)\). According to \Cref{thm:bipartite}, this task is equivalent to determining when \( q \neq 0 \).

\begin{question}
    For which classes of connected bipartite graphs is $q\neq 0$?
\end{question}

 In this paper, we identified several fundamental classes of graphs where  the degree of the $h$-polynomial is \(\alpha(G) \), achieving its maximum value. Our last question generalizes the first one and  invites researchers to identify other classes of graphs with this property.

\begin{question}
    For which classes of graphs is $\deg h_{R/I(G)} (t)=\alpha (G)$?
\end{question}

\section{Connected graphs with the maximum regularity and degree sum}\label{sec:classification}

In the final section of this paper, we consider the relationship between the degree of the $h$-polynomial and the regularity of edge ideals. In particular, we provide a characterization of connected graphs where the following upper bound from \citep[Theorem 13]{hibi2019regularity} is achieved: 
$$\reg R/I(G) + \deg h_{R/I(G)}(t) \leq |V(G)|.$$

This characterization is given in terms of Cameron-Walker graphs. We start the section by recalling an equivalent definition of these graphs.

\begin{thm}\citetext{\citealp[Theorem 1]{CameronWalker}}
    A graph $G$ is a Cameron-Walker graph, i.e. $\nu(G)=\mu(G)$ if and only if it is in one of the following forms (see Figure \ref{fig:CW1}):
\begin{itemize}
    \item a star graph;
    \item a $3$-cycle; or
    \item a graph consisting of a connected bipartite graph with a bipartition $(U,V)$, where $U=\{u_1,\ldots, u_r\}$ and $V=\{v_1,\ldots, v_d\},$ such that at least one leaf edge is attached to each vertex $u_i\in U$ and there may be  one or more pendant triangles attached to the vertices in $V$.
\end{itemize}

\begin{figure}[ht]
    \centering
    \begin{tikzpicture}[line cap=round,line join=round,>=triangle 45,x=1cm,y=1cm]
        \clip(-7.682920413991129,-2.823657521129758) rectangle (7.635113071340564,4.806051759510582);
        \draw [line width=1pt,dashed] (-4,2)-- (-1,2);
        \draw [line width=1pt,dashed] (-1,2)-- (4,2);
        \draw [line width=1pt,dashed] (-4,2)-- (-6,0);
        \draw [line width=1pt,dashed] (4,2)-- (6,0);
        \draw [line width=1pt,dashed] (-6,0)-- (-2,0);
        \draw [line width=1pt,dashed] (-2,0)-- (6,0);
        \draw [line width=1pt] (-4,2)-- (-5,4);
        \draw [line width=1pt] (-4,2)-- (-3,4);
        \draw [line width=1pt] (-1,2)-- (-2,4);
        \draw [line width=1pt] (-1,2)-- (0,4);
        \draw [line width=1pt] (4,2)-- (3,4);
        \draw [line width=1pt] (4,2)-- (5,4);
        \draw [line width=1pt] (-6,0)-- (-7.5,-1.05);
        \draw [line width=1pt] (-7.5,-1.05)-- (-7,-2);
        \draw [line width=1pt] (-7,-2)-- (-6,0);
        \draw [line width=1pt] (-6,0)-- (-5,-2);
        \draw [line width=1pt] (-5,-2)-- (-4.5,-0.95);
        \draw [line width=1pt] (-4.5,-0.95)-- (-6,0);
        \draw [line width=1pt] (-1,0)-- (-2.5,-1);
        \draw [line width=1pt] (-2.5,-1)-- (-2,-2);
        \draw [line width=1pt] (-2,-2)-- (-1,0);
        \draw [line width=1pt] (-1,0)-- (0,-2);
        \draw [line width=1pt] (0,-2)-- (0.5,-1);
        \draw [line width=1pt] (0.5,-1)-- (-1,0);

    \begin{small}
        \draw [fill=black] (-4,2) circle (2pt);
        \draw[color=black] (-3.99,1.7) node {$v_1$};
        \draw [fill=black] (-1,2) circle (2pt);
        \draw[color=black] (-1.01,1.7) node {$v_2$};
        \draw [fill=black] (4,2) circle (2pt);
        \draw[color=black] (3.99,1.7) node {$v_r$};
        \draw[color=black] (0,1) node {Connected bipartite graph on $\{u_1,\ldots, u_d\}\cup \{v_1,\ldots,v_r\}$};
        \draw [fill=black] (-6,0) circle (2pt);
        \draw[color=black] (-6.13,0.22) node {$u_1$};
        \draw [fill=black] (6,0) circle (2pt);
        \draw[color=black] (6.08,0.22) node {$u_d$};
        \draw [fill=black] (-1,0) circle (2pt);
        \draw[color=black] (-0.91,0.22) node {$u_k$};
        \draw [fill=black] (2,0) circle (2pt);
        \draw[color=black] (2.1,0.22) node {$u_{k+1}$};

        \draw [fill=black] (-5,4) circle (2pt);
        \draw[color=black] (-4.91,4.20) node {$u_{1,1}$};
        \draw [fill=black] (-3,4) circle (2pt);
        \draw[color=black] (-2.92,4.20) node {$u_{1,l_1}$};
        \draw [fill=black] (-2,4) circle (2pt);
        \draw[color=black] (-1.91,4.20) node {$u_{2,1}$};
        \draw [fill=black] (0,4) circle (2pt);
        \draw[color=black] (0.07,4.20) node {$u_{2,l_2}$};
        \draw [fill=black] (3,4) circle (2pt);
        \draw[color=black] (3.08,4.20) node {$u_{r,1}$};
        \draw [fill=black] (5,4) circle (2pt);
        \draw[color=black] (5.07,4.20) node {$u_{r,l_r}$};
        \draw [fill=black] (-7.5,-1.05) circle (2pt);
        \draw [fill=black] (-7,-2) circle (2pt);
        \draw[color=black] (-7,-1.2) node {$T_{1,1}$};
        \draw [fill=black] (-5,-2) circle (2pt);
        \draw[color=black] (-5,-1.2) node {$T_{1,t_1}$};
        \draw [fill=black] (-4.5,-0.95) circle (2pt);
        \draw [fill=black] (-2.5,-1) circle (2pt);
        \draw [fill=black] (-2,-2) circle (2pt);
        \draw[color=black] (-2,-1.2) node {$T_{k,1}$};
        \draw [fill=black] (0,-2) circle (2pt);
        \draw[color=black] (0,-1.2) node {$T_{k,t_k}$};
        \draw [fill=black] (0.5,-1) circle (2pt);

        % For the . . . between leaf vertices
        \draw [fill=black] (-4,3.75) circle (.5pt);
        \draw [fill=black] (-4.5,3.75) circle (.5pt);
        \draw [fill=black] (-3.5,3.75) circle (.5pt);
        \draw [fill=black] (-1,3.75) circle (.5pt);
        \draw [fill=black] (-1.5,3.75) circle (.5pt);
        \draw [fill=black] (-0.5,3.75) circle (.5pt);
        \draw [fill=black] (4,3.75) circle (.5pt);
        \draw [fill=black] (4.5,3.75) circle (.5pt);
        \draw [fill=black] (3.5,3.75) circle (.5pt);
         \draw [fill=black] (2,2.75) circle (.5pt);
          \draw [fill=black] (1.5,2.75) circle (.5pt);
           \draw [fill=black] (1,2.75) circle (.5pt);
        % For the . . . between triangles
        \draw [fill=black] (-6.5,-1.75) circle (.5pt);
        \draw [fill=black] (-6,-1.75) circle (.5pt);
        \draw [fill=black] (-5.5,-1.75) circle (.5pt);
        \draw [fill=black] (-1,-1.75) circle (.5pt);
        \draw [fill=black] (-1.5,-1.75) circle (.5pt);
        \draw [fill=black] (-0.5,-1.75) circle (.5pt);

        \draw [fill=black] (-4.25,-.5) circle (.5pt);
        \draw [fill=black] (-3.5,-.5) circle (.5pt);
        \draw [fill=black] (-2.75,-.5) circle (.5pt);
        \draw [decorate,decoration={brace,amplitude=5pt,mirror,raise=2ex}]
  (1.75,0) -- (6.25,0) node[midway,yshift=-2em]{Non-leaf vertices in $U$}node[midway,yshift=-3em]{that are pendant triangle free};
    \end{small}
\end{tikzpicture}
    \caption{Cameron-Walker Graphs}
    \label{fig:CW1}
\end{figure}
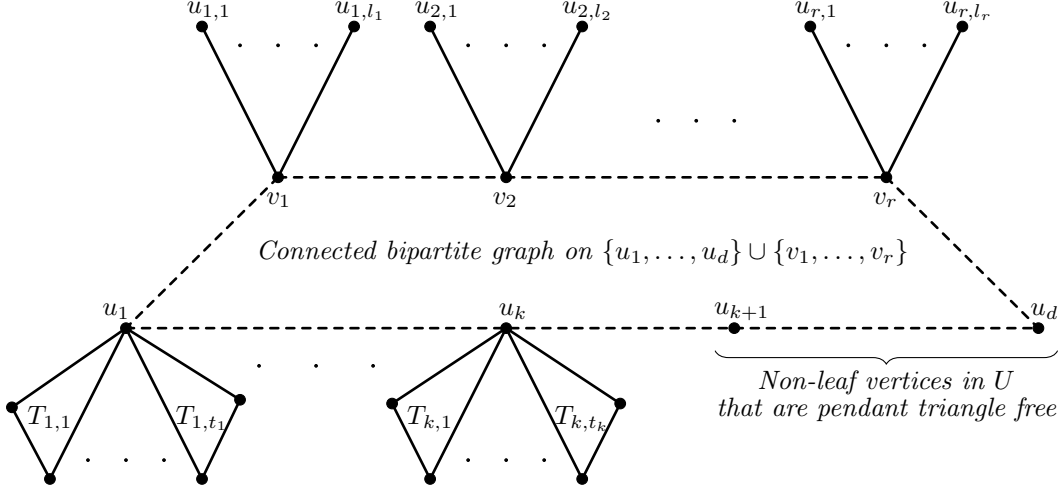

\end{thm}

\begin{notation}\label{notation:CW}
Let $G$ be a Cameron-Walker graph as in \Cref{fig:CW1}. Let $L_U$ denote all of the leaf neighbors of vertices in $V=\{v_1,\dots,v_r\}$, i.e. $L_U=\{u_{i,j}~:~1\leq i\leq r, 1\leq j\leq l_i\}.$ Let $T$ denote the set of all pendant triangles in $G$. We set $u_1,\dots, u_k$ to be the vertices with pendant triangles, and $u_{k+1}\dots,u_d$ to be the non-leaf vertices in $U$ that have no pendant triangle. 
\end{notation}

\begin{thm}\citetext{\citealp[Proposition 1.3]{CWRegDeg}; \citealp[Lemma 4.2]{hibi2021regularity}}\label{thm:CWregdeg}
Let $G$ be a Cameron-Walker graph. Adopting the notation set in \Cref{notation:CW} we have the following formulas for the degree of the $h$-polynomial and the regularity of a Cameron-Walker graph $G$. 
\begin{align*}
  \deg h_{R/I(G)} (t)&= |L_U|+ |T|+ (d-k)\\
  \reg R/I(G)&= r+ |T|  
\end{align*}   
\end{thm}

Note that with this notation, 
\[
|V(G)|=r+d+|L_U|+2|T|.    
\]
We are now ready to present main result of this section.

\begin{thm}\label{thm:reg_deg=n}
  Let $G$ be a connected graph. Then $$\reg R/I(G) + \deg h_{R/I(G)}(t) =|V(G)|$$  if and only if $G$ is  a Cameron-Walker graph with no pendant triangles.    
\end{thm}

\begin{proof}
    Let $G$ be a connected graph on $n$ vertices.  We investigate $\reg R/I(G) + \deg h_{R/I(G)}(t)$ by considering whether $G$ is a Cameron-Walker graph or not and repeatedly applying \Cref{thm:CWregdeg}.
    
Assume $G$ is a Cameron-Walker graph with notation as in \Cref{notation:CW}. If $G$ has no pendant triangles then
    \[
        \reg R/I(G) + \deg h_{R/I(G)} (t) =r+  |L_U|+d = |V(G)|.
    \]
If $G$ has pendant triangles, i.e. $k\geq 1$, then we have
    \begin{align*}
       \reg R/I(G) + \deg h_{R/I(G)} (t) &=(r+|T|)+(|L_U|+|T|+(d-k) ) \\
       &<r+d+|L_U|+2|T|\\
       &=|V(G)|. 
    \end{align*}
Now assume $G$ is not a Cameron-Walker graph. If $G$ is a 5-cycle, then $ \reg R/I(G)=2$  and $\deg h_{R/I(G)} (t)= 2$, by \Cref{thm:reg_match} and \Cref{thm:path_cycle}, respectively. So \[\reg R/I(G) + \deg h_{R/I(G)} (t)<5=\vert V(G) \vert .\]  

If $G$ is not a Cameron-Walker graph and not a 5-cycle then we have $\reg R/I(G) < \mu(G)$ from \Cref{thm:reg_match}. Thus \Cref{rmk:deg_alpha} and \Cref{thm:konig} tells us that \[\reg R/I(G) + \deg h_{R/I(G)} (t) < \mu(G)+\alpha(G)\leq \vert V(G)\vert\]
 concluding the proof.
\end{proof}

\bibliographystyle{abbrvnat}
\bibliography{references}

\begin{thebibliography}{16}
\providecommand{\natexlab}[1]{#1}
\providecommand{\url}[1]{\texttt{#1}}
\expandafter\ifx\csname urlstyle\endcsname\relax
  \providecommand{\doi}[1]{doi: #1}\else
  \providecommand{\doi}{doi: \begingroup \urlstyle{rm}\Url}\fi

\bibitem[Cameron and Walker(2005)]{CameronWalker}
K.~Cameron and T.~Walker.
\newblock The graphs with maximum induced matching and maximum matching the
  same size.
\newblock \emph{Discrete Math.}, 299:\penalty0 49--55, 2005.
\newblock \doi{10.1016/j.disc.2004.07.022}.

\bibitem[Dao and Schweig(2013)]{ProjDim}
H.~Dao and J.~Schweig.
\newblock Projective dimension, graph domination parameters, and independence
  complex homology.
\newblock \emph{J. Combin. Theory Ser. A}, 120\penalty0 (2):\penalty0 453--469,
  2013.
\newblock ISSN 0097-3165,1096-0899.
\newblock \doi{10.1016/j.jcta.2012.09.005}.
\newblock URL \url{https://doi.org/10.1016/j.jcta.2012.09.005}.

\bibitem[Favacchio et~al.(2020)Favacchio, Keiper, and
  Van~Tuyl]{favacchio2020regularity}
G.~Favacchio, G.~Keiper, and A.~Van~Tuyl.
\newblock Regularity and h-polynomials of toric ideals of graphs.
\newblock \emph{Proceedings of the American Mathematical Society}, 148\penalty0
  (11):\penalty0 4665--4677, 2020.

\bibitem[H{\`a} and Van~Tuyl(2008)]{ha2008monomial}
H.~T. H{\`a} and A.~Van~Tuyl.
\newblock Monomial ideals, edge ideals of hypergraphs, and their graded {B}etti
  numbers.
\newblock \emph{Journal of Algebraic Combinatorics}, 27:\penalty0 215--245,
  2008.

\bibitem[Herzog et~al.(2011)Herzog, Hibi, Herzog, and Hibi]{herzog2011monomial}
J.~Herzog, T.~Hibi, J.~Herzog, and T.~Hibi.
\newblock \emph{Monomial ideals}.
\newblock Springer, 2011.

\bibitem[Hibi and Matsuda(2018)]{hibi2018regularity}
T.~Hibi and K.~Matsuda.
\newblock Regularity and h-polynomials of monomial ideals.
\newblock \emph{Mathematische Nachrichten}, 291\penalty0 (16):\penalty0
  2427--2434, 2018.

\bibitem[Hibi et~al.(2019{\natexlab{a}})Hibi, Kanno, and Matsuda]{IndMatch}
T.~Hibi, H.~Kanno, and K.~Matsuda.
\newblock Induced matching numbers of finite graphs and edge ideals.
\newblock \emph{J. Algebra}, 532:\penalty0 311--322, 2019{\natexlab{a}}.
\newblock ISSN 0021-8693,1090-266X.
\newblock \doi{10.1016/j.jalgebra.2019.04.036}.
\newblock URL \url{https://doi.org/10.1016/j.jalgebra.2019.04.036}.

\bibitem[Hibi et~al.(2019{\natexlab{b}})Hibi, Matsuda, and
  Van~Tuyl]{hibi2019regularity}
T.~Hibi, K.~Matsuda, and A.~Van~Tuyl.
\newblock Regularity and $ h $-polynomials of edge ideals.
\newblock \emph{The Electronic Journal of Combinatorics}, pages P1--22,
  2019{\natexlab{b}}.

\bibitem[Hibi et~al.(2021{\natexlab{a}})Hibi, Kanno, Kimura, Matsuda, and
  Van~Tuyl]{hibi2021homological}
T.~Hibi, H.~Kanno, K.~Kimura, K.~Matsuda, and A.~Van~Tuyl.
\newblock Homological invariants of {C}ameron--{W}alker graphs.
\newblock \emph{Transactions of the American Mathematical Society},
  374\penalty0 (09):\penalty0 6559--6582, 2021{\natexlab{a}}.

\bibitem[Hibi et~al.(2021{\natexlab{b}})Hibi, Kimura, Matsuda, and
  Tsuchiya]{CWRegDeg}
T.~Hibi, K.~Kimura, K.~Matsuda, and A.~Tsuchiya.
\newblock Regularity and {$a$}-invariant of {C}ameron-{W}alker graphs.
\newblock \emph{J. Algebra}, 584:\penalty0 215--242, 2021{\natexlab{b}}.
\newblock ISSN 0021-8693,1090-266X.
\newblock \doi{10.1016/j.jalgebra.2021.05.007}.
\newblock URL \url{https://doi.org/10.1016/j.jalgebra.2021.05.007}.

\bibitem[Hibi et~al.(2022)Hibi, Kimura, Matsuda, and
  Van~Tuyl]{hibi2021regularity}
T.~Hibi, K.~Kimura, K.~Matsuda, and A.~Van~Tuyl.
\newblock The regularity and h-polynomial of {C}ameron--{W}alker graphs.
\newblock \emph{Enumerative Combinatorics and Applications}, 2\penalty0 (3),
  2022.

\bibitem[Katzman(2006)]{katzman2006characteristic}
M.~Katzman.
\newblock Characteristic-independence of {B}etti numbers of graph ideals.
\newblock \emph{Journal of Combinatorial Theory, Series A}, 113\penalty0
  (3):\penalty0 435--454, 2006.

\bibitem[Lov\'asz and Plummer(1986)]{LovaszPlummer1986}
L.~Lov\'asz and M.~D. Plummer.
\newblock \emph{Matching Theory}.
\newblock AMS Chelsea Publishing, 1986.

\bibitem[Matsumura(1987)]{matsumura1987}
H.~Matsumura.
\newblock \emph{Commutative Ring Theory}.
\newblock Cambridge Studies in Advanced Mathematics. Cambridge University
  Press, 1987.
\newblock \doi{10.1017/CBO9781139171762}.

\bibitem[Sturmfels(1996)]{Sturmfels1996}
B.~Sturmfels.
\newblock \emph{Gr\"obner Bases and Convex Polytopes}, volume~8 of
  \emph{University Lecture Series}.
\newblock American Mathematical Society, 1996.
\newblock \doi{10.1090/ulect/008}.

\bibitem[Trung(2020)]{RegMatch}
T.~N. Trung.
\newblock Regularity, matchings and {C}ameron-{W}alker graphs.
\newblock \emph{Collect. Math.}, 71\penalty0 (1):\penalty0 83--91, 2020.
\newblock ISSN 0010-0757,2038-4815.
\newblock \doi{10.1007/s13348-019-00250-9}.
\newblock URL \url{https://doi.org/10.1007/s13348-019-00250-9}.

\end{thebibliography}

\end{document}